\documentclass{article}
\usepackage{amssymb,amsmath,amsthm}
\usepackage[english]{babel}
\usepackage{MnSymbol}
\usepackage{t1enc}
\usepackage[latin2]{inputenc}
\usepackage{epsfig,afterpage,subfigure}
\usepackage{footnpag}
\usepackage{graphicx}
\usepackage[update,prepend]{epstopdf}
\usepackage{tikz}
\usetikzlibrary{positioning,shapes.multipart}
\usepackage{floatpag}
\floatpagestyle{empty} 
\usepackage{algorithm,algorithmicx}
\usepackage[noend]{algpseudocode}
\usepackage[nottoc,notlot,notlof]{tocbibind}

\theoremstyle{plain}
\newtheorem*{defi*}{Definition}
\newtheorem{thm}{Theorem}
\newtheorem{defi}[thm]{Definition}
\newtheorem{cor}[thm]{Corollary}
\newtheorem{lem}[thm]{Lemma}
\newtheorem{prop}[thm]{Proposition}
\newtheorem{claim}[thm]{Claim}

\newtheorem{problem}[thm]{Problem}
\newtheorem*{conj*}{Conjecture}

\theoremstyle{remark}
\newtheorem{remark}[thm]{Remark}
\newtheorem*{remark*}{Remark}

\usepackage{indentfirst}
\def\R{\mathbb R}
\def\N{\mathbb N}

\def\C{\mathcal C}
\let\hosszuo\H 
\def\H{\mathcal H} 

\def\eps{\varepsilon}

\newcommand{\remove}[1]{}

\begin{document}

\title{Unsplittable coverings in the plane}
\author{
J\'anos Pach\footnote{EPFL, Lausanne and R\'enyi Institute, Budapest.
Email: {\tt pach@cims.nyu.edu}.} \and
D\"om\"ot\"or P\'alv\"olgyi\footnote{Institute of Mathematics, E\"otv\"os University, Budapest.
Email: {\tt dom@cs.elte.hu}}}

\setcounter{footnote}{1}
\maketitle
\thispagestyle{empty}
\hfill {\em \small Et tu mi fili, Brute?}\footnote{The authors were completely convinced that the unit disk does not misbehave.\\
Research was supported by Hungarian Scientific Research Fund EuroGIGA Grant OTKA NN 102029 and PD 104386, by Swiss National Science Foundation Grants 200020-144531 and 200021-137574.
This work started in 1986, when the second author was still in kindergarden \cite{MP86}.
}

 \hfill \small{(Julius Caesar)}

\begin{abstract}
A system of sets forms an {\em $m$-fold covering} of a set $X$ if every point of $X$ belongs to at least $m$ of its members. A $1$-fold covering is called a {\em covering}. The problem of splitting multiple coverings into several coverings was motivated by classical density estimates for {\em sphere packings} as well as by the {\em planar sensor cover problem}. It has been the prevailing conjecture for 35 years (settled in many special cases) that for every plane convex body $C$, there exists a constant $m=m(C)$ such that every $m$-fold covering of the plane with translates of $C$ splits into $2$ coverings. In the present paper, it is proved that this conjecture is false for the unit disk. The proof can be generalized to construct, for every $m$, an unsplittable $m$-fold covering of the plane with translates of any open convex body $C$ which has a smooth boundary with everywhere {\em positive curvature}. Somewhat surprisingly, {\em unbounded} open convex sets $C$ do not misbehave, they satisfy the conjecture: every $3$-fold covering of any region of the plane by translates of such a set $C$ splits into two coverings. To establish this result, we prove a general coloring theorem for hypergraphs of a special type: {\em shift-chains}. We also show that there is a constant $c>0$ such that, for any positive integer $m$, every $m$-fold covering of a region with unit disks splits into two coverings, provided that every point is covered by {\em at most} $c2^{m/2}$ sets.
\end{abstract}



\clearpage
\section{Introduction}

Let $\C$ be a family of sets in $\mathbb R^d$, and let $P\subseteq\mathbb{R}^d$.
We say that ${\C}$ is an {\em $m$-fold covering of $P$} if every point of $P$ belongs to at least $m$ members of $\C$.
A $1$-fold covering is called 
a {\em covering}.
Clearly, the union of $m$ coverings is an $m$-fold covering.
We will be mostly interested in the case when $P$ is a large region or the whole space $\mathbb{R}^d$.

\smallskip

Sphere packings and coverings have been studied for centuries, partially because of their applications in crystallography, Diophantine approximation, number theory, and elsewhere. The research in this field has been dominated by density questions of the following type: What is the most ``economical'' (i.e., least dense) $m$-fold covering of space by unit balls or by translates of a fixed convex body? It is suggested by many classical results and physical observations that, at least in low-dimensional spaces, the optimal arrangements are typically periodic, and they can be split into several lattice-like coverings~\cite{Fej83,FejK93}. Does a similar phenomenon hold for all sufficiently ``thick'' multiple coverings, without any assumption on their densities?

\smallskip

About 15 years ago, a similar problem was raised for {\em large scale ad hoc sensor networks}; see~Feige {\em et al.}~\cite{FeH00}, Buchsbaum {\em et al.}~\cite{B07}. In the -- by now rather extensive -- literature, it is usually referred to as the {\em sensor cover problem}. In its simplest version it can be phrased as follows. Suppose that a large region $P$ is monitored by a set of sensors, each having a circular range of unit radius and each powered by a battery of unit lifetime. Suppose that every point of $P$ is within the range of at least $m$ sensors, that is, the family of ranges of the sensors, $\C$, forms an $m$-fold covering of $P$. If $\C$ can be split into $k$ coverings $\C_1,\ldots,\C_k$, then the region can be monitored by the sensors for at least $k$ units of time. Indeed, at time $i$, we can switch on all sensors whose ranges belong to $\C_i\; (1 \le i \le k)$. We want to maximize $k$, in order to guarantee the longest possible service. Of course, the first question is the following.

\begin{problem}[Pach, 1980~\cite{P80}]\label{thm:question}
Is it true that every $m$-fold covering of the plane with unit disks splits into two coverings, provided that $m$ is sufficiently large?
\end{problem}

\begin{figure}
\begin{center}
\addtolength{\subfigcapskip}{.3cm}
\subfigure[The disks form a $2$-fold covering of the green triangle. The colors give a split into $2$ coverings. It follows, however, from the proof of Theorem~\ref{thm:disc}, that for any $m$ there is an $m$-fold covering of a triangle by disks that does not split.]{\label{fig:diskcover}
\includegraphics[width=.35\textwidth]{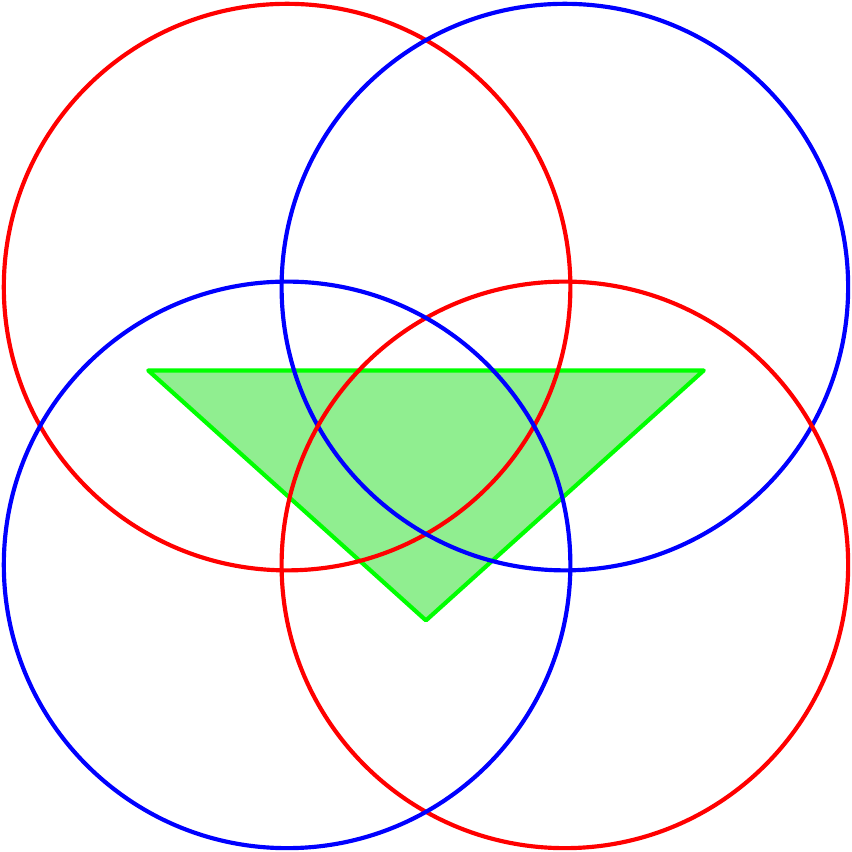}}
\hfill
\subfigure[The parabolas form a $2$-fold covering of the green triangle, but no matter how we $2$-color them, there will be a point not covered by one of the colors. It follows, however, from Theorem~\ref{thm:unbounded} and a standard compactness argument, that any $3$-fold covering of a closed triangle by the translates of an open parabola splits into $2$ coverings.]{\label{fig:parabolacover}
\includegraphics[width=.5\textwidth]{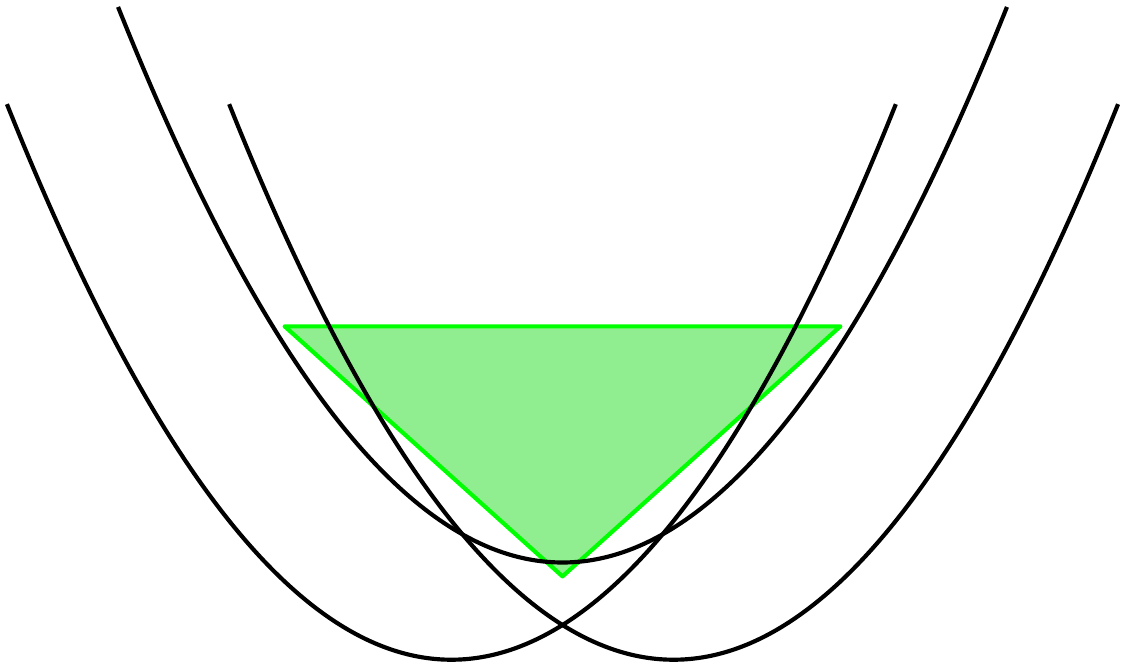}}
\caption{Two simple examples.}\label{fig:simple}
\end{center}
\end{figure}

In a long unpublished manuscript, Mani and Pach~\cite{MP86} claimed that the answer to this question was in the affirmative with $m\le 33$. Pach~\cite{PT09} warned that this ``has never been independently verified.'' Winkler~\cite{W09} even conjectured that the statement is true with $m=4$. For more than 30 years, the prevailing conjecture has been that for any open plane {\em convex body} (i.e., bounded convex set) $C$, there exists a positive integer $m=m(C)$ such that every $m$-fold covering of the plane with translates of $C$ splits into two coverings. This conjecture was proved in~\cite{P86} for centrally symmetric convex polygons $C$. It took almost 25 years to generalize this statement to all convex polygons~\cite{TT07,PT10}. Moreover, it was proved by Aloupis {\em et al.}~\cite{A09} and Gibson and Varadarajan~\cite{GV09} that in these cases, for every integer $k$, every at least $bk$-fold covering splits into $k$ coverings, where $b=b(C)$ is a suitable positive constant. See~\cite{PPT,PhD,PTT09}, for surveys.

Here we disprove the above conjecture by giving a negative answer to Problem~\ref{thm:question}.

\begin{thm}\label{thm:disc}
For every positive integer $m$, there exists an $m$-fold covering of the plane with open unit disks that cannot be split into $2$ coverings.
\end{thm}


Our construction can be generalized as follows.

\begin{thm}\label{thm:smooth} Let $C$ be any open plane convex set, which has two parallel supporting lines with positive curvature at their points of tangencies. Then, for every positive integer $m$, there exists an $m$-fold covering of the plane with translates of $C$ that cannot be split into $2$ coverings.
\end{thm}

As was mentioned above, for every open convex polygon $Q$, there exists a smallest positive integer $m(Q)$ such that every $m(Q)$-fold covering of the plane with translates of $Q$ splits into $2$ coverings. We have that $\sup m(Q)=\infty$, where the sup is taken over all convex polygons $Q$. Otherwise, we could approximate the unit disk with convex $n$-gons with $n$ tending to infinity. By compactness, we would conclude that the unit disk $C$ satisfies $m(C)<+\infty$, which contradicts Theorem~\ref{thm:disc}.

\begin{problem}\label{thm:4szog} Does there exist, for any $n>3$, an integer $m(n)$ such that every convex $n$-gon $Q$ satisfies $m(Q)\le m(n)$?
\end{problem}

For any triangle $T$, there is an affine transformation of the plane that takes it into an equilateral triangle $T_0$. Therefore, we have $m(T)=m(T_0)$ and $m(3)$ is finite. For $n=4$, Problem~\ref{thm:4szog} is open.

\smallskip

In spite of our sobering negative answer to Problem~\ref{thm:question} and its analogues in higher dimensions (cp.~\cite{MP86}), there are important classes of multiple coverings such that all of their members are splittable. According to our next, somewhat counter-intuitive result, for example, any $m$-fold covering of $\mathbb{R}^d$ with unit balls can be split into $2$ coverings, provided that no point of the space is covered by too many balls. (We could innocently believe that heavily covered points make it only easier to split an arrangement.)

\begin{thm}\label{thm:lll} For every $d\ge 2$, there exists a positive constant $c_d$ with the following property. For every positive integer $m$, any $m$-fold covering of $\mathbb{R}^d$ with unit balls can be split into two coverings, provided that no point of the space belongs to more than $c_d2^{m/d}$ balls.
\end{thm}

Theorem~\ref{thm:lll} was one of the first geometric applications of the Lov\'asz local lemma~\cite{LLL}, and it was included in~\cite{alonspencer}. Here, we establish a more general statement (see Theorem~\ref{shatter}).

One may also believe that {\em unbounded} convex sets behave even worse than the bounded ones. It turns out, however, that this is not the case.

\begin{thm}\label{thm:unbounded} Let $C$ be an unbounded open convex set and let $P$ be a finite set of points in the plane. Then every $3$-fold covering of $P\subset\mathbb{R}^2$ with translates of $C$ can be split into two coverings of $P$.
\end{thm}

In fact, using a standard compactness argument, Theorem~\ref{thm:unbounded} also holds if $P$ is any {\em compact} set in the plane.
However, Theorem~\ref{thm:unbounded} does not generalize to higher dimensions. Indeed, it follows from the proof of Theorem~\ref{thm:disc} that, for every positive integer $m$, there exists a finite family $\C$ of open unit disks in the plane and a finite set $P\subset \mathbb{R}^2$ such that $\C$ is an $m$-fold covering of $P$ that cannot be split into two coverings. Consider now an unbounded convex cone $C'$ in $\mathbb{R}^3$, whose intersection with the plane $\mathbb{R}^2$ is an open disk. Take a system of translates of $C'$ such that their intersections with the plane coincide with the members of $\C$. These cones form an $m$-fold covering of $P$ that cannot be split into two coverings.

For interesting technical reasons, the proof of Theorem~\ref{thm:unbounded} becomes much easier if we restrict our attention to multiple coverings of the {\em whole plane}. In fact, in this case, we do not even have to consider {\em multiple} coverings! Moreover, the statement remains true in higher dimensions.

\begin{prop}\label{claim}
Let $C$ be an unbounded line-free open convex set in $\mathbb{R}^d$. 
Then every covering of $\mathbb{R}^d$ with translates of $C$ can be split into two, and hence into infinitely many, coverings.
\end{prop}

The reason why we assume here that $C$ is {\em line-free} (i.e., does not contain a full line) is the following.
If $C$ contains a straight line, then it can be obtained as the direct product of a line $l$ and a $(d-1)$-dimensional open convex set $C'$. Any arrangement $\C$ of translates of $C$ in $\mathbb{R}^d$ is combinatorially equivalent to the $(d-1)$-dimensional arrangement of translates of $C'$, obtained by cutting $\C$ with a hyperplane orthogonal to $l$. In particular, the problem whether an $m$-fold covering of $\mathbb{R}^d$ with translates of $C$ can be split into two coverings reduces to the respective question about $m$-fold coverings of $\mathbb{R}^{d-1}$ with translates of $C'$.

Proposition~\ref{claim} is false already in the plane without the assumption that $C$ is open. 
However, every $2$-fold covering of the plane with translates of an unbounded $C$ can be split into two coverings.
We omit the proof as it reduces to a simple claim about intervals. 



However, in higher dimensions, the similar claim is false. 

\begin{thm}\label{ugly} There is a bounded convex set $C'\subset \mathbb{R}^3$ with the following property. One can construct a family of translates of $C=C'\times [0,\infty)\subset\mathbb{R}^4$ which covers every point of $\mathbb{R}^4$ infinitely many times, but which cannot be split into two coverings.
\end{thm}

The construction given in Section~\ref{sec:unbounded} is based on an example of Nasz\'odi and Taschuk \cite{NT}, and explores the fact that the boundary of $C'$ can be rather ``erratic.''
We do not know whether sufficiently thick coverings of $\R^3$ by translates of an unbounded line-free convex set can be split into two coverings or not.

\bigskip

In the sequel, we will study the equivalent {\em ``dual'' form} of the above questions. Consider a family $\C=\{C_i : i\in I\}$ of translates of a set $C\subset\mathbb{R}^d$ that form an $m$-fold covering of $P\subseteq \mathbb{R}^d$. Suppose without loss of generality that $C$ contains the origin $0$. For every $i\in I$, let $c_i$ denote the point of $C_i$ that corresponds to $0\in C$. In other words, we have $\C=\{ C+c_i : i\in I\}$. Assign to each $p\in P$ a translate of $-C$, the reflection of $C$ about the origin, by setting $C^*_p=-C+p$. Observe that $$p\in C_i \;\; \Longleftrightarrow \;\; c_i\in C^*_p.$$
In particular, the fact that $\C$ forms an $m$-fold covering of $P$ is equivalent to the following property: Every member of the family $\C^*=\{C^*_p : p\in P\}$ contains at least $m$ elements of $\{c_i : i\in I\}$. Thus, Theorem~\ref{thm:disc} can be rephrased in the following {\em dual} form.

\medskip 
\noindent{\bf Theorem~\ref{thm:disc}'.} {\em For every $m\ge 2$, there is a set of points $P^*=P^*(m)$ in the plane with the property that every open unit disk contains at least $m$ elements of $P^*$, and no matter how we color the elements of $P^*$ with two colors, there exists a unit disk such that all points in it are of the same color.}
\medskip

A set system {\em not} satisfying this condition is said to have {\em property B} (in honor of Bernstein) or is
{\em $2$-colorable} (see \cite{Mi37,E63,RS}). Generalizations of this notion are related to {\em conflict-free colorings} \cite{E02} and have strong connections, e.g., to the theory of {\em $\eps$-nets}, {\em geometric set covers} and to {\em combinatorial game theory} \cite{HW,PTT09,Alon10,V10,G09}.

The rest of this paper is organized as follows.
In the next three sections, we prove Theorem~\ref{thm:disc}' in $3$ steps.
In Section~\ref{sec:hypergraph}, we exhibit a family of non-$2$-colorable $m$-uniform hypergraphs $\H(k,l)$.
In Section~\ref{sec:finite}, we construct planar ``realizations'' of these hypergraphs, where the vertices correspond to points and the (hyper)edges to unit disks, preserving the incidence relations.
In Section~\ref{sec:plane}, we extend this construction, without violating the colorability condition, so that every disk contains at least $m$ points.
In Section~\ref{sec:smooth}, we modify these steps in order to establish Theorem~\ref{thm:smooth}, a generalization of Theorem~\ref{thm:disc} to bounded plane convex bodies with a smooth boundary.
Sections~\ref{sec:shiftchain} and~\ref{sec:unbounded} contain the proofs of our results related to multiple coverings with {\em unbounded} convex sets: Theorem~\ref{thm:unbounded}, Proposition~\ref{claim}, and Theorem~\ref{ugly}.
The proof of a more general version of Theorem~\ref{thm:lll}, using the Lov\'asz local lemma, can be found in Section~\ref{sec:lll}.
Finally, in Section~\ref{sec:open} we make some concluding remarks and mention a couple of open problems.

\numberwithin{thm}{section}
\section{A family of non-$2$-colorable hypergraphs $\H(k,l)$}\label{sec:hypergraph}

In this section we define, for any positive integers $k$ and $l$, an abstract hypergraph $\H(k,l)$ with vertex set $V(k,l)$ and edge set $E(k,l)$.
The hypergraphs $\H(k,l)$ are defined recursively.
The edge set $E(k,l)$ will be the disjoint union of two sets, $E(k,l)=E_R(k,l)\cupdot E_B(k,l)$, where the subscripts $R$ and $B$ stand for red and blue.
All edges belonging to $E_R(k,l)$ will be of size $k$, all edges belonging to $E_B(k,l)$ will be of size $l$.
In other words, $\H(k,l)$ is the union of a {\em $k$-uniform} and an {\em $l$-uniform} hypergraph.
If $k=l=m$, we get an $m$-uniform hypergraph.

\begin{defi} Let $k$ and $l$ be positive integers.
\begin{enumerate}
\item For $k=1$, let $V(1,l)$ be an $l$-element set.\\ Set $E_R(1,l):=V(1,l)$ and $E_B(1,l):=\{V(1,l)\}$.

\item For $l=1$, let $V(k,1)$ be a $k$-element set.\\ Set $E_R(k,1):=\{V(k,1)\}$ and $E_B(k,1):=V(k,1)$.


\item For any $k,l>1$, we pick a new vertex $p$, called the {\em root}, and let
$$V(k,l):=V(k-1,l)\cupdot V(k,l-1)\cupdot\{p\},$$
$$E_R(k,l):=\{e\cup \{p\} \,:\, e\in E_R(k-1,l)\} \cup E_R(k,l-1),$$
$$E_B(k,l):=E_B(k-1,l)\cupdot\{e\cup \{p\} \,:\, e\in E_B(k,l-1)\}.$$
\end{enumerate}
\end{defi}

\begin{figure}[t]
\begin{center}
\includegraphics[width=10cm]{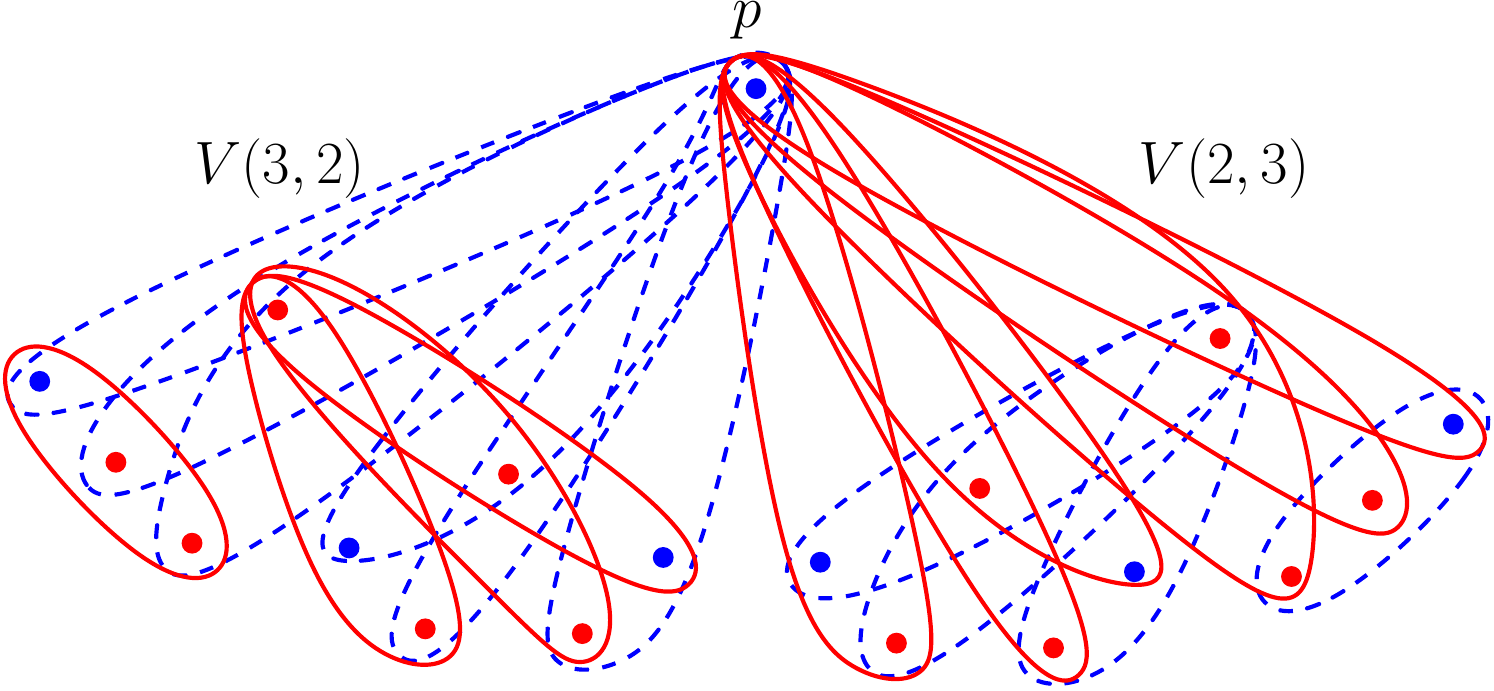}
\caption{The hypergraph $\H(3,3)$ with (arbitrarily) 2-colored vertices. There is a blue (dashed) set with $3$ blue vertices or a red (solid) set with $3$ red vertices.}\label{fig:halmazrendszer}
\end{center}
\end{figure}

By recursion, we obtain that
$$|V(k,l)|={k+l\choose k}-1,$$
$$|E_R(k,l)|={k+l-1\choose k},\,\, |E_B(k,l)|={k+l-1\choose l},$$
$$|E(k,l)|=|E_R(k,l)|+|E_B(k,l)|= {k+l\choose k}.$$


\begin{lem}[\cite{P10}]\label{hipergraf}
For any positive integers $k, l$, the hypergraph $\H(k,l)$ 
is not $2$-colorable. Moreover, for every coloring of $V(k,l)$ with red and blue, there is an edge in $E_R(k,l)$ such that all of its $k$ vertices are red or an edge in $E_B(k,l)$ such that all of its $l$ vertices are blue.
\end{lem}

For completeness, here we include the proof of Lemma~\ref{hipergraf} from \cite{P10}.
The induction on two parameters, $k$ and $l$, is similar to the proof of Ramsey's theorem by Erd\hosszuo os and Szekeres \cite{ESz}.

\begin{proof}
We will prove that for every coloring of $V(k,l)$ with red and blue, there is an edge in $E_R(k,l)$ such that all of its $k$ vertices are red or an edge in $E_B(k,l)$ such that all of its $l$ vertices are blue.

Suppose first that $k=1$. If any vertex in $V(1,l)$ is red, then it is a red singleton edge in $\H(1,l)$. If all vertices in $V(1,l)$ are blue, then the (only) edge $V(1,l)\in E_B(1,l)$ contains only blue points. Analogously, the assertion is true if $l=1$.

Suppose next that $k, l >1$. Assume without loss of generality that the root $p$ is red. Consider the subhypergraph $\H(k-1,l)\subset\H(k,l)$ induced by the vertices in $V(k-1,l)$. If it has a monochromatic red edge $e\in E_R(k-1,l)$, then $e\cup\{p\}\in E_R(k,l)$ is red. If there is a monochromatic blue edge in $E_B(k-1,l)$, then we are again done, because it is also an edge in $E_B(k,l)$.
\end{proof}

\noindent  For other interesting properties of the hypergraphs $\H(k,l)$ related to hereditary discrepancy, see Matou\v sek \cite{M11}.

\section{Geometric realization of the hypergraphs $\H(k,l)$}\label{sec:finite}

The aim of this section is to establish the following weaker version of Theorem~\ref{thm:disc}'.

\medskip
\noindent{\bf Theorem~\ref{thm:disc}''.} {\em For every $m\ge 2$, there exists a {\em finite} point set $P=P(m)\subset\mathbb{R}^2$ and a {\em finite} family of unit disks $\C=\C(m)$ with the property that every member of $\C$ contains at least $m$ elements of $P$, and no matter how we color the elements of $P$ with two colors, there exists a disk in $\C$ such that all points in it are of the same color.}
\medskip


\smallskip
We {\em realize} the hypergraph $\H(k,l)$ defined in Section~\ref{sec:hypergraph} with points and disks.
The vertex set $V(k,l)$ is mapped to a point set $P(k,l)\subset\mathbb{R}^2$, and the edge sets, $E_R(k,l)$ and $E_B(k,l)$, to families of open unit disks, $\C_R(k,l)$ and $\C_B(k,l)$, so that a vertex belongs to an edge if and only if the corresponding point is contained in the corresponding disk.
The geometric properties of this realization are summarized in the following lemma.

Given two unit disks $C,C'$, let $d(C,C')$ denote the distance between their centers.
We fix an orthogonal coordinate system in the plane so that we can talk about the {\em topmost} and the {\em bottommost} points of a disk.

\begin{lem}\label{lem:construction} For any positive integers $k, l$ and for any $\eps>0$, there is a finite point set $P=P(k,l)$ and a finite family of open unit disks $\C(k,l)=\C_R(k,l)\cupdot\C_B(k,l)$ with the following properties.
\begin{enumerate}
\item Any disk $C\in\C_R(k,l)$ (resp.\ $\C_B(k,l)$) contains precisely $k$ (resp.\ $l$) points of $P$.
\item For any coloring of $P$ with red and blue, there is a disk in $\C_R(k,l)$ such that all of its points are red or a disk in $\C_B(k,l)$ such that all of its point are blue.
In fact, $P$ and $\C(k,l)$ {\em realize} the abstract hypergraph $\H(k,l)$ in the above sense.
\item For the coordinates $(x,y)$ of any point from $P$, we have $-\eps<x<\eps$ and $-\eps^2<y<\eps^2$. 
\item For the coordinates $(x,y)$ of the center of any disk from $\C_R(k,l)$, we have $-\eps<x<\eps$ and $-\eps^2<y-1<\eps^2$.
\item For the coordinates $(x,y)$ of the center of any disk from $\C_B(k,l)$, we have $-\eps<x<\eps$ and $-\eps^2<y+1<\eps^2$.
\item The topmost and the bottommost points of a disk $C\in\C(k,l)$ are not covered by the closure of any other member of $\C(k,l)$.
\end{enumerate}
\end{lem}

Looking at our construction from ``far away'' the two families $\C_R$ and $\C_B$ look like two touching disks, with all points of $P$ very close to the touching point.
The segments connecting the centers of disks from different families are almost vertical with all members of $\C_R$ lying ``above'' all members of $\C_B$.
We prove the lemma by induction.
Most conditions are needed for the induction to go through.
Condition 6 is an exception: it will be used in Section~\ref{sec:plane}.

\begin{figure}
\vspace*{-2cm}
\begin{center}
\subfigure[Starting step: $\C(k,1)$.]{\label{fig:start}
\includegraphics[width=.5\textwidth]{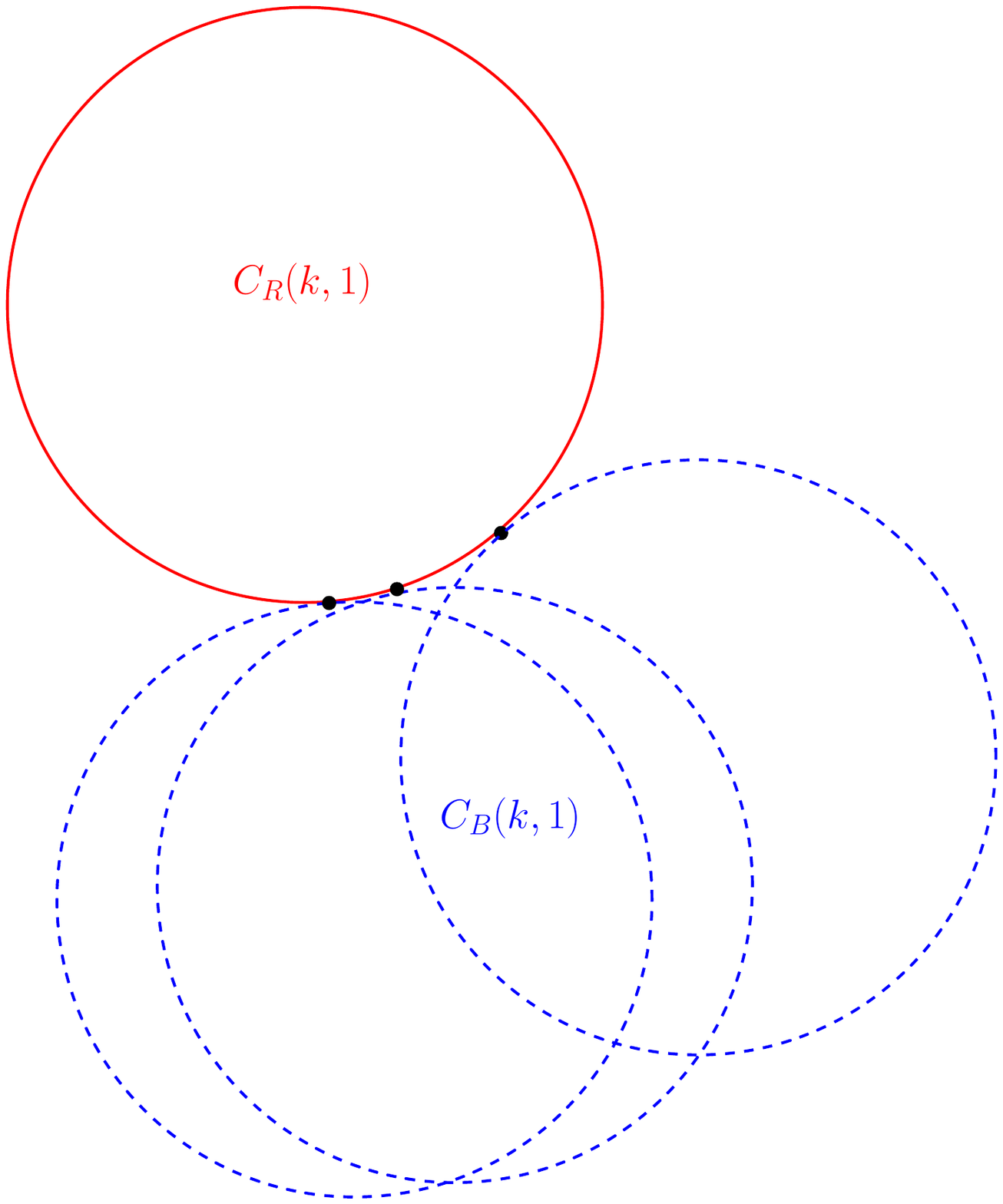}}
\hfill
\subfigure[$\C(2,2)$ magnified (and a bit distorted for visibility).]{\label{fig:22}
\includegraphics[width=.4\textwidth]{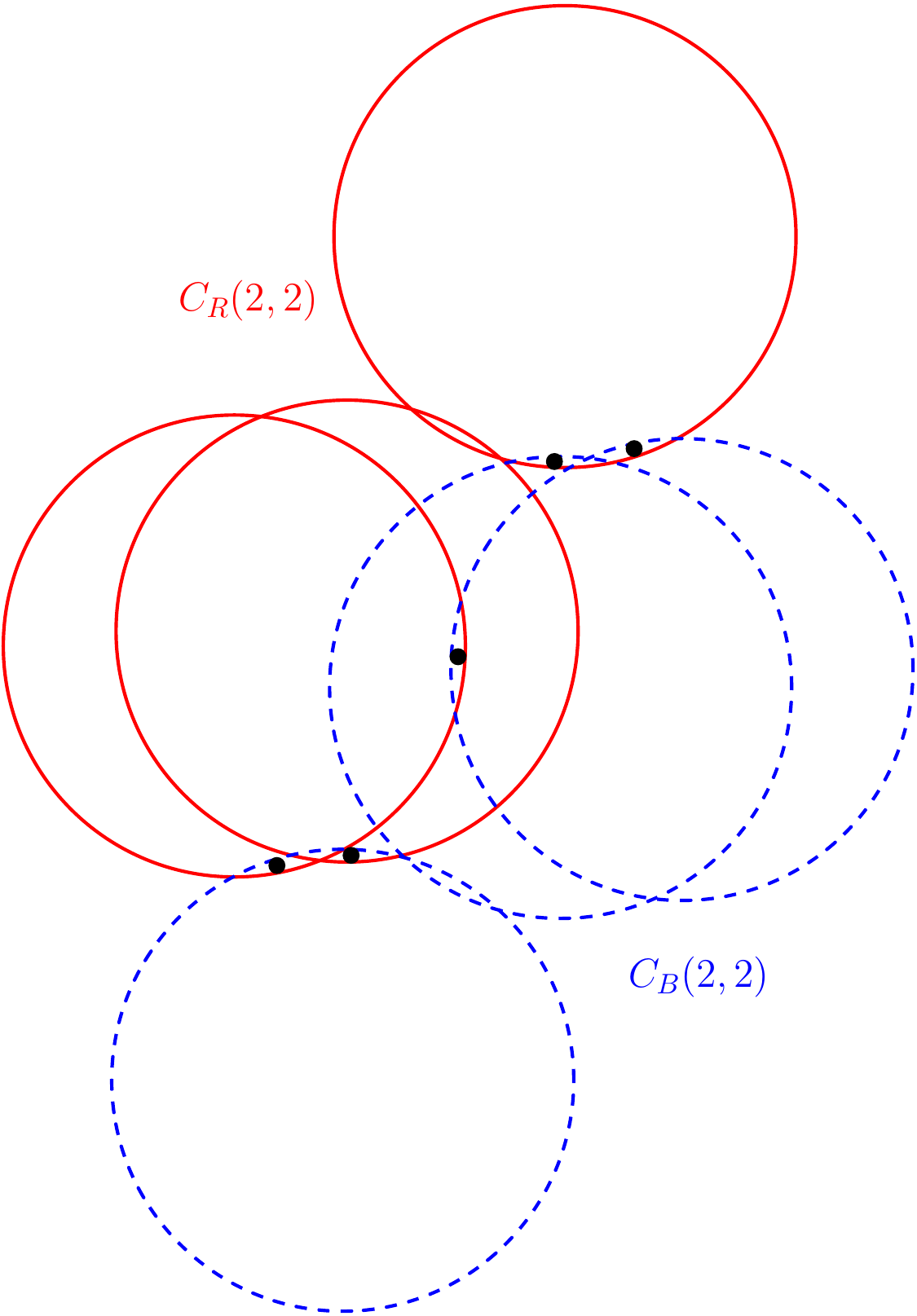}}
\subfigure[Induction step.]{\label{fig:ind}
\includegraphics[width=.9\textwidth]{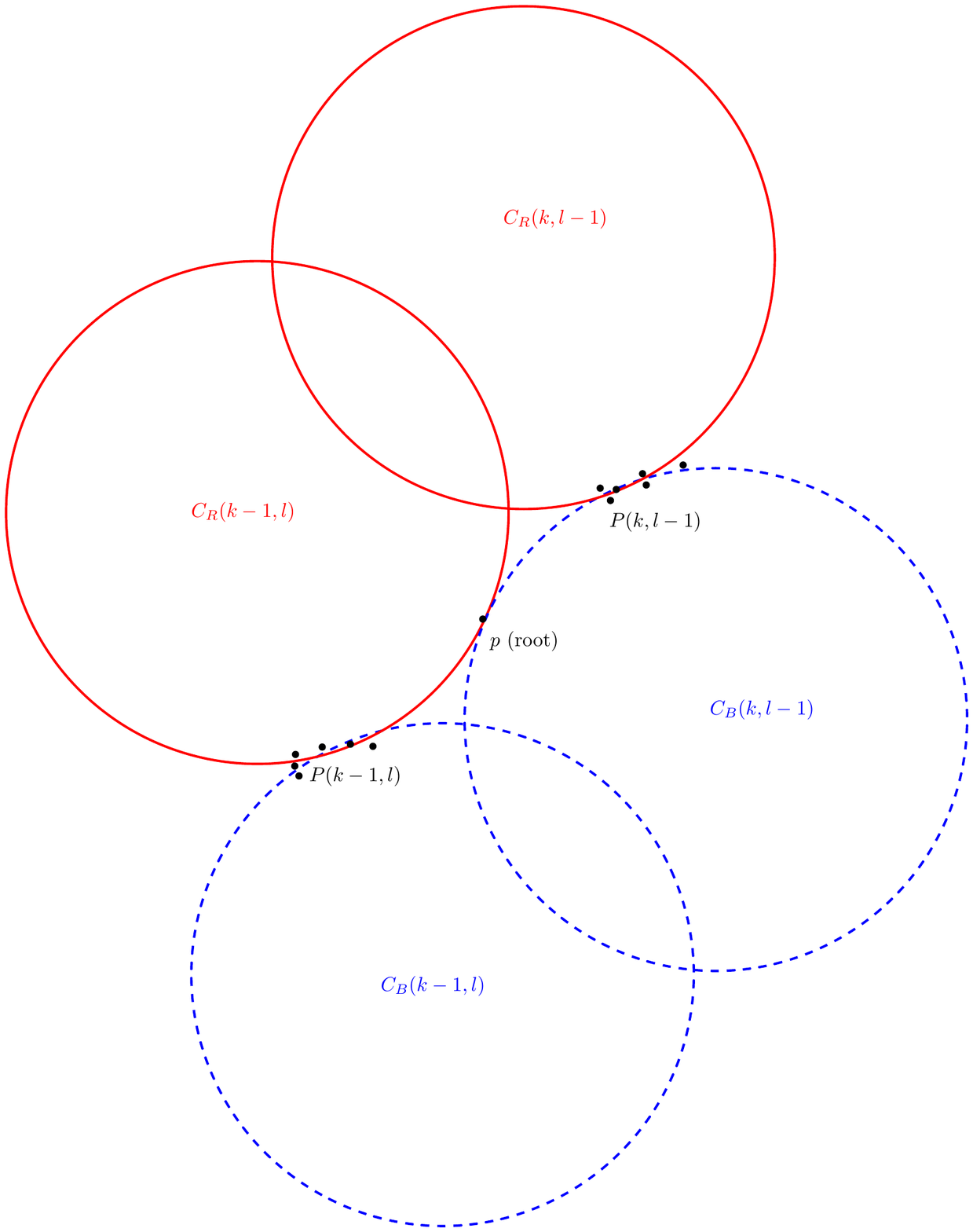}}
\caption{The construction.}\label{allfigs}
\end{center}
\end{figure}



\begin{proof}
We give a recursive construction.
We can assume that $\eps<1/10$.
It is easy to see that, for $k=1$ or $l=1$, there exists such a family of unit disks for any $\eps>0$, see Figure~\ref{fig:start}.
The family $\C(2,2)$ is depicted in Figure~\ref{fig:22}, where the main idea of the induction may already be visible.

Suppose that $k,l\ge 2$ and we have already constructed $P(k-1,l)$ and $\C(k-1,l)$, and $P(k,l-1)$ and $\C(k,l-1)$, for some $\eps(k-1,l)<\eps/100$ and $\eps(k,l-1)<\eps/100$, respectively.
To obtain $P(k,l)$, we place the root $p$ of $\H(k,l)$ into the origin $(0,0)$,
and we shift (translate) $P(k-1,l)$ and $P(k,l-1)$ into new positions such that their roots are at $(-\eps/3, -\eps^2/10)$ and $(\eps/3, \eps^2/10)$, respectively. With a slight abuse of notation, the shifted copies will also be denoted $P(k-1,l)$ and $P(k,l-1)$.
See Figure~\ref{fig:ind}.
In this way, it is guaranteed that for the coordinates $(x,y)$ of any point of $P$, we have
$$-\eps<-(\eps/3+\eps(k-1,l)+\eps(k,l-1))<x<\eps/3+\eps(k-1,l)+\eps(k,l-1)<\eps$$
and
$$-\eps^2<-(\eps^2/10+\eps^2(k-1,l)+\eps^2(k,l-1))<y<\eps^2/3+\eps^2(k-1,l)+\eps^2(k,l-1)<\eps^2.$$
Thus, {\bf property 3} of the lemma holds.

The family $\C(k,l)$ is defined as the union of two previously defined families, $\C(k-1,l)$ and $\C(k,l-1)$, translated by the same vectors as $P(k-1,l)$ and, resp.\ $P(k,l-1)$ were. Again, we use the same symbols to denote the translated copies.
To verify {\bf properties 4 and 5}, we only have to repeat the above calculations, with the $y$-coordinates being shifted $1$ higher (resp.\ $1$ lower).


Now we show that our set of points $P(k,l)$ and set of disks $\C(k,l)$ realize the hypergraph $\H(k,l)$ ({\bf properties 1 and 2}).
It is easy to see that if $C\in \C_R(k-1,l)$ and $s\in P(k,l-1)$, then $s\notin C$ but $p=(0,0)\in C$.
The coordinates of the center of $C$ are $\big(-\eps/3\pm \eps(k-1,l), 1-\eps^2/10\pm \eps^2(k-1,l)\big)$ (where here and in the following, $\pm z$ denotes a number that is between $-z$ and $z$), so the distance of $p$ from $C$ is at most $(\eps/3+\eps(k-1,l))^2+(1-\eps^2/10+\eps^2(k-1,l))^2<1$.
On the other hand, the coordinates of $s$ are
$\big(\eps/3\pm \eps(k,l-1), \eps^2/10\pm \eps^2(k,l-1)\big)$, thus the square of its distance from the center of $C$ is at least
$$\big(2\eps/3-\eps(k-1,l)-\eps(k,l-1)\big)^2+\big(1-2\eps^2/10-\eps^2(k-1,l)-\eps^2(k,l-1)\big)^2>1.$$
Analogously, if $C\in \C_B(k,l-1)$ and $s\in P(k-1,l)$, then $s\notin C$ but $p=(0,0)\in C$.

Let $C\in \C_R(k,l-1)$ and $s\in P(k-1,l)$.
We prove that $p, s\notin C$.
The coordinates of the center of $C$ are
$\big(\eps/3\pm \eps(k,l-1), 1+\eps^2/10\pm \eps(k,l-1)\big)$.
Therefore, the distance of $p$ from the center of $C$ is at least $(\eps/3-\eps(k,l-1))^2+(1+\eps^2/10-\eps(k,l-1))^2>1$.
The calculation for $s$ is similar in the case $C\in \C_R(k-1,l)$.
Analogously, we have that if $C\in \C_B(k-1,l)$ and $s\in P(k,l-1)$, then $p, s\notin C$.
As the disks in $\C(k,l-1)$ (resp.\ $\C(k-1,l)$) contain precisely the same points of $P(k,l-1)$ (resp.\ $P(k-1,l)$, as before the shift, we have obtained a geometric realization of $\H(k,l)$, and properties 1 and 2 hold.

It remains to prove that the topmost and the bottommost points of a disk $C\in\C(k,l)$ are not covered by any other member of $\C(k,l)$ ({\bf property 6}).
Using that our construction and disks are centrally symmetric, it is enough to prove the statement for the topmost points.
If $C\in\C_R(k,l-1)$, the coordinates of its topmost point are
$\big(\eps/3\pm \eps(k,l-1), 2+\eps^2/10\pm \eps^2(k,l-1)\big)$.
If $C\in\C_R(k-1,l)$, the coordinates of its topmost point are
$\big(-\eps/3\pm \eps(k-1,l), 2-\eps^2/10\pm \eps^2(k-1,l)\big)$.
If $C\in\C_B(k,l-1)$, the coordinates of its topmost point are
$\big(\eps/3\pm \eps(k,l-1), -2+\eps^2/10\pm \eps^2(k,l-1)\big)$.
If $C\in\C_B(k-1,l)$, the coordinates of its topmost point are
$\big(-\eps/3\pm \eps(k-1,l), -2-\eps^2/10\pm \eps^2(k-1,l)\big)$.

If $C\in\C_R(k,l-1)$, by the induction hypothesis, its topmost point cannot be covered by any other disk from $\C(k,l-1)$.
Nor can it be covered by any other disk, as the topmost points of all other disks are below it (i.e., have smaller $y$-coordinates).
If $C\in\C_R(k-1,l)$, then the square of the distance of its topmost point from the center of some $C'\in \C_R(k,l-1)$ is at least
$$\big(2\eps/3 -\eps(k,l-1)-\eps(k-1,l)\big)^2 + \big(1-2\eps^2/10 -\eps^2(k,l-1)-\eps^2(k-1,l)\big)^2>1.$$
If $C\in\C_B(k,l-1)$, then the distance of its topmost point from the center of some $C'\in \C_R(k-1,l)$
is also at least
$$\big(2\eps/3 -\eps(k,l-1)-\eps(k-1,l)\big)^2 + \big(1-2\eps^2/10 -\eps^2(k,l-1)-\eps^2(k-1,l)\big)^2>1.$$
In all other cases, trivially, the corresponding distances are also larger than $1$. 
This completes the proof of property 6 and hence the lemma.
\end{proof}


\section{Adding points to $P$ -- Proof of Theorem~\ref{thm:disc}'}\label{sec:plane}

In this section, we extend the proof of Theorem~\ref{thm:disc}'' to establish Theorem~\ref{thm:disc}' (which is equivalent to Theorem~\ref{thm:disc}). Note that the only difference between Theorems~\ref{thm:disc}'' and~\ref{thm:disc}' is that in the latter it is also required that every unit disk of the plane contains at least $m$ elements of the point set $P^*=P^*(m)$. The set $P=P(m,m)$ constructed in Lemma~\ref{lem:construction}, does not satisfy this condition. In order to fix this, we will add all points {\em not} in $\cup \C(m,m)$ to the set $P$ (or rather a sufficiently dense discrete subset of $\mathbb{R}^2\setminus\cup \C(m,m)$). In order to show that the resulting set $P^*$ meets the requirements of Theorem~\ref{thm:disc}', all we have to show is the following.

\begin{lem}\label{lem:kieg}
No (open) unit disk $C\notin \C(k,l)$ is entirely contained in $\cup \C(k,l)$.
\end{lem}

For future purposes, we prove this statement in a slightly more general form. In what follows, we only assume that $C$ is an open convex body with a unique topmost point $t$ and a unique bottommost point $b$, which divide the boundary of $C$ into two {\em closed arcs}. They will be referred to as the {\em left boundary arc} and a {\em right boundary arc}.

\begin{defi}\label{def:exposed} A collection $\C$ of translates of $C$ is said to be {\em exposed} if the topmost and bottommost points of its members do not belong to the closure of any other member of $\C$.
\end{defi}

By the last condition in Lemma~\ref{lem:construction}, the collections of disks $\C(k,l)$ constructed in the previous section are exposed. We prove the following generalization of Lemma~\ref{lem:kieg}.

\begin{lem}\label{lem:kieg'}
Let $\C$ be a finite exposed collection of translates of an open convex body $C$ with unique topmost and bottommost points. If $C\notin\C$, then $C\not\subseteq\cup\C$.
\end{lem}

For the proof, we need a simple observation.

\begin{claim}\label{obs:nonsym} If the right boundary arcs of two translates of $C$ intersect, then
the closure of one of the translates must contain the topmost or bottommost point of the other.
\end{claim}
\begin{proof}
Let $C_1$ and $C_2$ be the two translates, and let $\gamma_i$ denote the closed convex curve formed by the right boundary arc of $C_i$ and the straight-line segment connecting its two endpoints (the topmost and the bottommost points of $C_i$). The curves $\gamma_1$ and $\gamma_2$ are translates of each other, and since they intersect, they must {\em cross} twice. 
(At a {\em crossing}, one curve comes from the exterior of the other, then it shares an arc with it, which may be a single point, and enters the interior.)
It cannot happen that both crossings occur between the right boundary arcs, because they are convex and translates of each other.
Therefore, one of the two crossings involves the straight-line segment of one the curves, say, $\gamma_1$.
But since the condition is that the right boundary arcs intersect, one of the two endpoints of this straight-line segment, either the topmost or the bottommost point of $C_1$, lies in the closure of $C_2$
\end{proof}

\begin{proof}[Proof of Lemma~\ref{lem:kieg'}.] Suppose, for contradiction, that $C\subseteq\cup\C$. By removing some members of $\C$ if necessary, we can assume that $\C$ is a {\em minimal} collection of translates that covers $C$. Then $C$ must have a point which belongs to (at least) three translates, $C_1, C_2, C_3\in\C$.
None of the topmost and bottommost points of these translates can be covered by $C$, otherwise, it would also be covered by another member of $\C$, contradicting the assumption that $\C$ is exposed.

Thus, $C$ intersects either the left or the right boundary arc of every $C_i$.
Without loss of generality, suppose that $C$ intersects the right boundary arcs of $C_1$ and $C_2$.
These right boundary arcs must intersect inside $C$, otherwise $C_1\cap C\subseteq C_2\cap C$ or $C_2\cap C\subseteq C_1\cap C$, and $\C$ would not be minimal.
Therefore, we can apply Claim~\ref{obs:nonsym} to conclude that one of them must contain the topmost or bottommost point of the other.
\end{proof}

\begin{remark} In the construction described in Lemma~\ref{lem:construction}, every disk in $\C(m,m)$ contains at most $|P(m,m)|<2^{2m}$ points. At the last stage, we added many new points to $P$. We can keep the maximum number of points of $P$ lying in a unit disk bounded from above by a function $f(m)$. What is the best upper bound? The bound given by our construction depends on $\eps(m,m)\le 100^{-2m}\eps(1,1)$.
\end{remark}

\section{Other convex bodies -- Proof of Theorem~\ref{thm:smooth}}\label{sec:smooth}

Throughout this section, $C$ denotes an open plane convex body which has two parallel supporting lines with positive curvature at the two points of tangencies. To prove Theorem~\ref{thm:smooth}, by duality, it is sufficient to establish the analogue of Theorem~\ref{thm:disc}', where the role of unit disks is played by translates of $C$.

\medskip
\noindent{\bf Theorem~\ref{thm:smooth}'.} {\em For every $m\ge 2$, there is a set of points $P^*=P^*(m)$ in the plane with the property that every translate of $C$ contains at least $m$ elements of $P^*$, and no matter how we color the elements of $P^*$ with two colors, there exists a translate of $C$ such that all points in it are of the same color.}
\medskip

As in the case of disks, after defining the hypergraphs $\H(k,l)$, the proof consists of two steps:

\smallskip
{\bf Step 1:} We find a geometric realization of $\H=\H(k,l)$ with translates of $C$, i.e., a finite point set $P$ representing the vertices and a collection $\C$ of translates of $C$ representing the hyperedges of $\H$ such that a point of $P$ lies in a member of $\C$ if and only if the corresponding vertex belongs to the corresponding hyperedge.
We show that $\C$ is an exposed family.

\smallskip
{\bf Step 2:} We show that no translate of $C$ is entirely contained in $\cup\C$, unless $C\in\C$. Thus, we can add all the points not in $\cup \C(k,l)$ to the points of $P$ to ensure that every translate of $C$ contains many points.
\smallskip

In Section~\ref{sec:plane}, we have shown that Step 2 can be completed, provided that $\C$ is exposed (see Lemma~\ref{lem:kieg'}). Therefore, here we concentrate on Step 1.

Without loss of generality, we can assume that $C$ has unique bottommost and topmost points, $b$ and $t$, resp., at which the curvature is positive.
After applying an affine transformation, we can also attain that the line $bt$ is vertical. 
Let $r_b$ and $r_t$ denote the reciprocals of the curvatures at $b$ and $t$, respectively.
If we place $b$ at the origin, then, for every $\delta>0$, in a small neighborhood of $b$, the boundary of $C$ will lie between the parabolas $y=(1-\delta)r_bx^2$ and $y=(1+\delta)r_bx^2$.
Analogously, if we place $t$ at the origin, then in a small neighborhood of it, the boundary of $C$ will lie between the parabolas $y=-(1-\delta)r_tx^2$ and $y=-(1+\delta)r_tx^2$.
We find a geometric realization using the following lemma.

\begin{figure}[h]
\centering
\includegraphics[width=10cm]{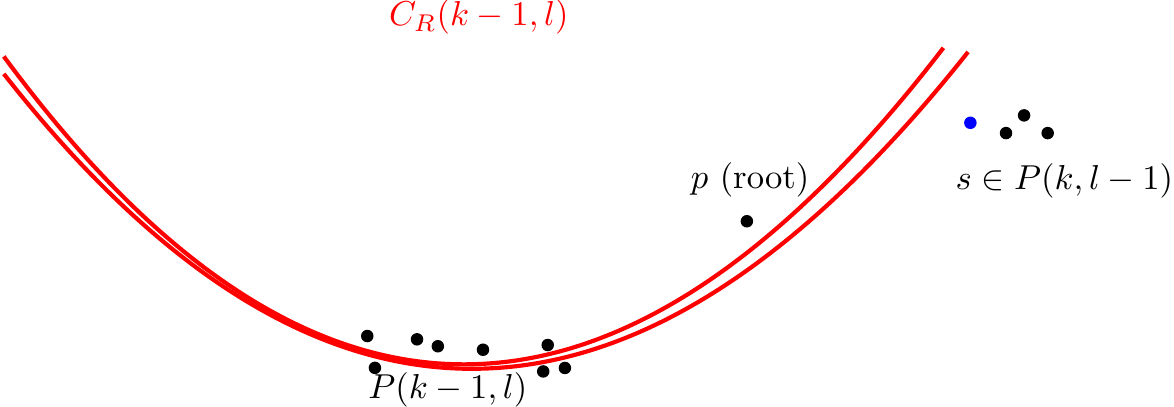}
\caption{Parabolas enclosing the boundary of $C$.}\label{fig:smooth}
\end{figure}

\medskip
\noindent{\bf Lemma~\ref{lem:construction}'.} {\em For any positive integers $k, l$ and for any $\eps>0$, there is a finite point set $P=P(k,l)$ and a finite family of translates of $C$, $\C(k,l)=\C_R(k,l)\cupdot\C_B(k,l)$ with the following properties.
\begin{enumerate}
\item Any translate from $\C_R(k,l)$ (resp.\ $\C_B(k,l)$) contains precisely $k$ (resp.\ $l$) points of $P$.
\item For any coloring of $P$ with red and blue, there is a translate from $\C_R(k,l)$ such that all of its points are red or a translate from $\C_B(k,l)$ such that all of its point are blue.
In fact, $P$ and $\C(k,l)$ {\em realize} the abstract hypergraph $\H(k,l)$ in the above sense.
\item For the coordinates $(x,y)$ of any point from $P$, we have $-\eps<x<\eps$ and $-\eps^2<y<\eps^2$. 
\item For the coordinates $(x,y)$ of the bottommost point of any translate from $\C_R(k,l)$, we have $-\eps<x<\eps$ and $-\eps^2<y<\eps^2$.
\item For the coordinates $(x,y)$ of the topmost point of any translate from $\C_B(k,l)$, we have $-\eps<x<\eps$ and $-\eps^2<y<\eps^2$.
\item The topmost and the bottommost points of translate from $\C(k,l)$ are not covered by the closure of any other member of $\C(k,l)$.
\end{enumerate}
}

\begin{proof}
Using an affine transformation, we can suppose that $r_t,r_b<1$.
We fix a $\delta$ that is small enough compared to $r_t$ and $r_b$, and an $\eps=\eps(k,l)$ that is small enough compared to $\delta$, $r_t$ and $r_b$, but big enough compared to $\eps(k,l-1)$ and $\eps(k-1,l)$.
(To keep the presentation simple, we omit the exact required dependencies here.)
We will use that the boundary of $C$ in a $2\eps(r_t+r_b)$ neighborhood around $t$ and $b$ is between the (above mentioned) pairs of parabolas, $y=(1-\delta)r_bx^2$ and $y=(1+\delta)r_bx^2$, and $y=-(1-\delta)r_tx^2$ and $y=-(1+\delta)r_tx^2$.

If $k=1$ or $l=1$, the construction is trivial.
For $k,l\ge 2$, assume that the point sets $P(k-1,l)$ and $P(k,l-1)$, and the families of translates of $C$, $\C(k-1,l)$ and $\C(k,l-1)$, have already been defined, and that they satisfy all conditions in the lemma.
To obtain $P(k,l)$, we place the root $p$ of $\H(k,l)$ into the origin $(0,0)$, and we shift $P(k-1,l)$ and $P(k,l-1)$ such that their roots are at $(-r_t\eps, -(1+2\delta)r_br_t^2\eps^2)$ and $(r_b\eps, (1+2\delta)r_tr_b^2\eps^2)$, respectively.
The family of translates $\C(k,l)$ is defined as the union of the families $\C(k-1,l)$ and $\C(k,l-1)$ translated by the same vectors, as $P(k-1,l)$ and $P(k,l-1)$, respectively.

To verify {\bf properties 3, 4, and 5}, we need that $-\eps<-r_t\eps,r_b\eps<\eps$ and $-\eps^2<-(1+2\delta)r_br_t^2\eps^2,(1+2\delta)r_tr_b^2\eps^2<\eps^2$, which hold since $r_t,r_b<1$ and $\delta$ is small.
Notice that where we have omitted $\eps(k,l-1)$ and $\eps(k-1,l)$ from these equations to keep the calculations simple.
This we can do as the difference of the two sides depends on $\eps$, which we can select to be sufficiently large compared to $\eps(k,l-1)$ and $\eps(k-1,l)$.
We will also omit dependencies of $\eps(k,l-1)$ and $\eps(k-1,l)$ later.

To verify {\bf properties 1 and 2}, we have to show that for any $C\in \C_R(k-1,l)$, the origin $p=(0,0)$ belongs to $C$, but no point $s\in P(k,l-1)$ does, provided that $\eps>0$ is sufficiently small.
To see this, fix $C\in \C_R(k-1,l)$.
The equation of the parabola that touches $C$  from the inside at its bottommost point is approximately $y=(1+\delta)r_b(x+r_t\eps)^2-(1+2\delta)r_br_t^2\eps^2$.
If $x=0$, 
the value of $y$ is $(1+\delta)r_b(r_t\eps)^2-(1+2\delta)r_br_t^2\eps^2=-\delta r_br_t^2\eps^2$.
This is negative, which means that $p=(0,0)$ lies above the parabola.
Thus, we have $p\in C$.
Analogously, if $C\in \C_B(k,l-1)$ and $s\in P(k-1,l)$, then $s\notin C$ but $p=(0,0)\in C$.

On the other hand, the equation of the parabola that touches $C$ at its bottommost point from the outside is approximately
$y=(1-\delta)r_b(x+r_t\eps)^2-(1+2\delta)r_br_t^2\eps^2$.
If $x=r_b\eps\pm \eps(k,l-1)$ 
the value of $y$ at $x$ is approximately
$$(1-\delta)r_b(r_b\eps+r_t\eps)^2-(1+2\delta)r_br_t^2\eps^2=\left((1-\delta) (r_b^3+2r_b^2r_t)-3\delta r_br_t^2\right)\eps^2
\ge \left(r_b^3+O(\delta)\right)\eps^2.$$
Therefore, $s=(r_b\eps\pm \eps(k,l-1), (1+2\delta)r_tr_b^2\eps^2\pm \eps^2(k,l-1))$ is below the parabola, if $\delta$ is small enough, thus $s\notin C$.

Let $C\in \C_R(k,l-1)$ and $s\in P(k-1,l)$.
We prove that $p, s\notin C$.
The equation of the parabola that touches $C$ from the outside at its bottommost point is approximately $y=(1-\delta)r_t(x-r_b\eps)^2-(1+2\delta)r_tr_b^2\eps^2$.
If $x=0$, the value of $y$ is $(1+\delta)r_t(-r_b\eps)^2-(1+2\delta)r_tr_b^2\eps^2=-\delta r_br_t^2\eps^2<0$, thus $p\in C$.
The calculation for $s$ is similar in the case $C\in \C_R(k-1,l)$.
Analogously, we have that if $C\in \C_B(k-1,l)$ and $s\in P(k,l-1)$, then $p, s\notin C$.
As the translates in $\C(k,l-1)$ (resp.\ $\C(k-1,l)$) contain precisely the same points of $P(k,l-1)$ (resp.\ $P(k-1,l)$, as before the shift, we have obtained a geometric realization of $\H(k,l)$, and properties 1 and 2 hold.

It remains to prove that the topmost and the bottommost points of a translate $\C(k,l)$ are not covered by any other member of $\C(k,l)$ ({\bf property 6}).
Using that our construction is symmetric, it is enough to prove the statement for the topmost points.
Recall that the line connecting $b$ and $t$ is vertical and denote their distance, the height of $C$, by $h$.  

The coordinates of the topmost points of translates from $\C_R(k,l-1)$ are approximately
$(r_b\eps+h, (1+2\delta)r_tr_b^2\eps^2+h)$.
The coordinates of the topmost points of translates from $\C_R(k-1,l)$ are approximately
$(-r_t\eps+h, -(1+2\delta)r_br_t^2\eps^2+h)$.
The coordinates of the topmost points of translates from $\C_B(k,l-1)$ are approximately
$(r_b\eps, (1+2\delta)r_tr_b^2\eps^2)$.
The coordinates of the topmost points of translates from $\C_B(k-1,l)$ are approximately
$(-r_t\eps, -(1+2\delta)r_br_t^2\eps^2)$.

If $C_1\in \C_R(k,l-1)$, by the induction hypothesis, its topmost point cannot be covered by any other $C_2\in \C(k,l-1)$.
Nor can it be covered by any other translate, as the topmost points of all other translates are below it (i.e., have smaller $y$-coordinates).
If $C_1\in\C_R(k-1,l)$, then the vector connecting it to the topmost point of some $C_2\in \C_R(k,l-1)$ is approximately the same as the vector connecting a point $s\in P(k-1,l)$ to the topmost point of some $C'\in \C_B(k,l-1)$.
As we have seen earlier that $s\notin C'$, the same calculation shows that the topmost point of $C_1$ is not in $C_2$.
If $C_1\in\C_R(k-1,l)$ and $C_2\in\C_B(k,l-1)$ or $C_2\in\C_B(k-1,l)$, then the topmost point of $C_2$ is lies below the topmost point of $C_1$.
If $C_1\in\C_B(k,l-1)$ and $C_2\in \C(k,l-1)$, by induction the topmost point of $C_1$ is not in $C_2$.
If $C_1\in\C_B(k,l-1)$ and $C_2\in \C(k-1,l)$, then the topmost point of $C_1$ is approximately at the same place as the points of $P(k,l-1)$ which are avoided by $C_2$, and the same calculation works here.
Similarly, if $C_1\in\C_B(k-1,l)$ and $C_2\in \C(k-1,l)$, we can use induction, and if $C_1\in\C_B(k-1,l)$ and $C_2\in \C(k,l-1)$, we can use that the topmost point of $C_1$ is approximately at the same place as the points of $P(k-1,l)$ which are avoided by $C_2$, the same calculation works here.
This completes the proof of property 6 and hence the lemma.
\end{proof}

\section{Shift-chains -- Proof of Theorem~\ref{thm:unbounded}}\label{sec:shiftchain}

Throughout this section, $P$ denotes a fixed set of $n$ points in the plane, no two of which have the same $x$-coordinate, and $C$ is a fixed open convex set that contains a vertical upward half-line.

\begin{defi}
For $A\subset [n]=\{1,2,\ldots,n\}$, denote by $a_i$ the $i^{th}$ smallest element of $A$.
For two equal sized sets, $A,B\subset [n]$, we write $A\preceq B$ if $a_i\le b_i$ for every $i$.

An $m$-uniform hypergraph on the vertex set $[n]$ is called a {\em shift-chain} if its hyperedges are totally ordered by the relation $\preceq$.
A shift-chain ${\mathcal H}$ is {\em special} if for any two hyperedges, $A,B\in\H$ with $A\preceq B$, we have $\max(A \setminus B)<\min(B\setminus A)$.
\end{defi}

\noindent For any integer $m$ and real number $x$, let $C(m;x)$ denote the translate of $C$ which

 {\bf a.} contains exactly $m$ points of $P$,

 {\bf b.} can be obtained from $C$ by translating it through a vector with $x$-coordinate $x$,

 {\bf c.} and has minimum $y$-coordinate, among all translates satisfying {\bf a} and {\bf b}.

\noindent The union of all translates of $C$ through every vector that has $x$-coordinate $x$ is a vertical strip (or an open half-plane or the whole plane), denoted by $S(x)$. If $S(x)$ contains precisely $m$ points for some $x$, then in condition {\bf c}, the minimum $y$-coordinate is $y=-\infty$, and we set $C(m;x)=S(x)$. If $S(x)$ contains fewer than $m$ points, then $C(m;x)$ is undefined.

\begin{prop}\label{obs:geom2shift} Let $p_1, p_2, \ldots, p_n$ denote the elements of $P$, listed in the increasing order of their $x$-coordinates.
Then the $m$-uniform hypergraph consisting of the sets $P(x)=\{ i\in [n] \, ; \, p_i\in C(m;x)\}$, over all $x\in\mathbb{R}$, is a special shift-chain.
\end{prop}
\begin{proof}
Notice that if 
$x<x'$, then the boundary of $C(m;x)$ intersects the boundary of $C(m;x')$ precisely once. Therefore, every element of $(C(m;x)\setminus C(m;x'))\cap P$ is to the left of all elements of $(C(m;x')\setminus C(m;x))\cap P$. This means that $P(x)\preceq P(x')$.
\end{proof}

In view of the duality described at the end of the introduction, Theorem~\ref{thm:unbounded} is an immediate corollary of the following statement.

\begin{thm}\label{thm:shiftchain} For any $m\ge 3$, every $m$-uniform special shift-chain is $2$-colorable. Moreover, such a coloring can be constructed in linear time.
\end{thm}

An example found by Fulek~\cite{F10} (depicted on 
Figure~\ref{fig:rado}) shows that Theorem~\ref{thm:shiftchain} is false without assuming that the shift-chain is special.

\begin{figure}[H]
\centering
\includegraphics[width=7cm]{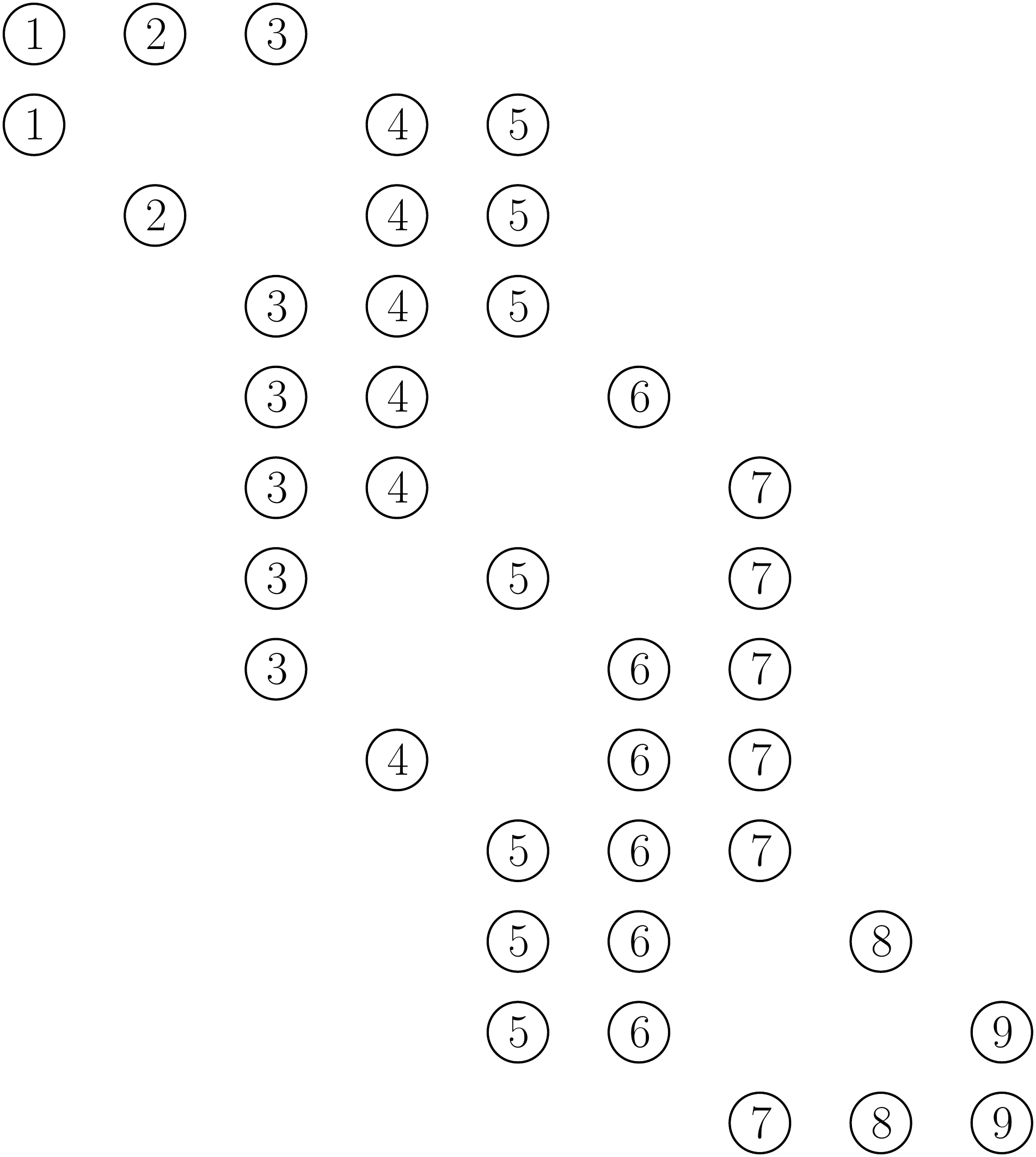} 
\caption{A shift-chain of $13$ triples, each of which corresponds to a row.
For any $2$-coloring of the $9$ vertices, one of the triples is monochromatic.}\label{fig:rado}
\end{figure}

\begin{problem}
Does there exist an integer $m_0>3$ such that for every $m\ge m_0$, every $m$-uniform shift-chain is $2$-colorable?
\end{problem}

If the answer to this question is yes, in some sense this could be regarded as an extension of the Lov\'asz local lemma~\cite{LLL}. For more problems and results related to shift-chains and special shift-chains, consult~\cite{PhD,KP14}.

\begin{proof}[Proof of Theorem~\ref{thm:shiftchain}.]
The proof breaks into several simple claims.
In the rest of this section, $\H$ denotes a fixed $3$-uniform special shift-chain on $[n]=\{1,2,\ldots,n\}$. For simplicity, a hyperedge (triple) $\{a,b,c\}\in\H$ with $a<b<c$ will be denoted by $\{a<b<c\}$.

\begin{claim}\label{graf} If $\{a<b<c\}\in \H$ and $\{a'<b<c'\}\in \H$, then $a'=a$ or $c'=c$.
\end{claim}
\begin{proof} Otherwise, $\{a<b<c\}\setminus \{a'<b<c'\}=\{a<c\}$ and $\{a'<b<c'\}\setminus \{a<b<c\}=\{a'<c'\}$ would not be separated, 
contradicting our assumption that $\H$ is special.
\end{proof}

Define a digraph, $D=D(\H)$ with vertex set $[n]$ and edge set $E$, as follows.
For any $b<c$, the directed edge $bc\in E$ if and only if there exist $a, a'\in [n]$, $a\ne a'$, such that $\{a<b<c\}\in \H$ and $\{a'<b<c\} \in \H$. Analogously, for any $a<b$, the directed edge $ba\in E$ if and only if there exist $c,c'\in [n]$, $c\ne c'$, such that $\{a<b<c\}\in \H$ and $\{a<b<c'\}\in \H$.
According to Claim~\ref{graf}, the out-degree of every vertex of $D$ is at most one. Note that an edge may appear in $E$ with both orientations $ab$ and $ba$.

\begin{claim}\label{graph} The directed graph $D$ can be constructed by a linear time algorithm.
\end{claim}
\begin{proof} $\H$, as any $3$-uniform shift-chain on $n$ vertices has at most $3n-8$ hyperedges. Suppose that they are listed in an arbitrary order, and process them one-by-one. Suppose the next triple is $\{a<b<c\}$.
\begin{enumerate}
\item If $b$ is a middle vertex for the first time, store it, together with both of its {\em neighbors}, $a$ and $c$.

\item If $b$ is a middle vertex for the second time, decide if it was $a$ or $c$ that has been previously stored as one of its neighbors. (By Claim~\ref{graf}, we know that one them was.) If it is $a$, add $ba$ to $E$, if it is $c$, add $bc$.

\item Otherwise, do not add any new edge, and pass to the next triple.\qedhere 
\end{enumerate}
\end{proof}

\begin{claim}\label{abc} For $a<b<c$ (or $c<b<a$) it is not possible that $ac\in E$ and $ba\in E$.
\end{claim}
\begin{proof} Suppose $ac, ba\in E$. By definition, this means that there exist $\{x<a<c\}\in\H$ and $\{a<b<y\}\in\H$, for some $x$ and some $y\ne c$. Obviously, with respect to the ordering of the triples, we have $\{x<a<c\}\prec \{a<b<y\}$. The sets $\{x<a<c\}\setminus\{a<b<y\}=\{x<c\}$ and $\{a<b<y\}\setminus\{x<a<c\}=\{b<y\}$ are not separated, because the maximal element of the first set, $c$, is larger than the minimal element of the second set, $b$. This contradicts our assumption that $\H$ was a {\em special} shift-chain.
\end{proof}

\begin{claim}\label{abcd} For $a<b<c<d$ it is not possible that $bd\in E$ and $ca\in E$.
\end{claim}
\begin{proof} This would mean that there exist $\{x<b<d\}\in\H$ with $x\ne a$ and $\{a<c<y\}\in\H$ with $y\ne d$. These two triples are disjoint, but not separated, contradicting the assumption that $\H$ is special.
\end{proof}

If a directed graph $T$ can be obtained from a directed tree oriented toward its root $r$, by possibly adding one of the edges $pr$ entering the root also with the reverse orientation $rp$, then it is called a {\em quasi-tree}. Note that in this case, we can also think of $T$ as a quasi-tree rooted in $p$.

\begin{claim} The graph $D$ is the vertex-disjoint union of quasi-trees.  
\end{claim}
\begin{proof}
As no vertex of $D$ has out-degree larger than $1$, it is enough to show
that $D$ has no directed cycle of length larger than $2$. Suppose there is such a directed cycle, and denote its smallest and largest elements by $a$ and $b$, respectively. By Claim~\ref{abc}, we have that $ab\notin E$ and $ba\notin E$. Let $ya\in E$ and $ax\in E$ be the incoming edge and the outgoing edge of the cycle at $a$. Again, by Claim~\ref{abc}, we have $a<x<y<b$. There is a directed path from $x$ to $b$, and along this path there is a first edge $uv$ with $u<y$ and $v>y$. But then the edges $ya, uv\in E$ would contradict Claim~\ref{abcd}, as $a<u<y<v$.
\end{proof}

Now we are in a position to find a $2$-coloring of $\H$ in linear time. 
For every $\{a<b<c\}\in \H$, we will guarantee that the color of its middle vertex, $b$, will differ from the color of $a$ or the color of $c$.

First, using breadth-first search, we properly $2$-color each connected component.
Hence, it will be guaranteed that if the out-degree of $b$ is non-zero, then all triples of the form $\{a<b<c\}\in \H$ contain both colors.
Then assign to each vertex $x\in[n]$ an edge $ab\in E$ such that $a< x<b$ or $b< x<a$, provided that such an edge exists.
This can be done in linear time, but here we omit the details.

For every $\{x<y<z\}\in \H$ that does not yet contain both colors, its middle vertex $y$ has out-degree zero.
If there is an edge $bc\in E$ assigned to $y$ such that $b< y<c$, then there are two different hyperedges $\{a<b<c\},\{a'<b<c\}\in \H$.
Either $a\ne x$ or $a'\ne x$, and thus, necessarily, we have $c=z$.
We color $y$ with the same color as $b$ (i.e., differently from $c=z$), so that $\{x<y<z\}$ contains both colors.
Note that if the in-degree of $y$ is non-zero, then a simple case analysis shows that the only possibility is $zy\in E$. Thus, this color agrees with the color given earlier to $y$ from its connected component.

In a similar manner, if there is an edge $ab\in E$ assigned to $y$ such that $a< y<b$, then necessarily $a=x$, and we can color $y$ with the same color as $b$ (i.e., differently from $a=x$).

Finally, if there are uncolored vertices, color them in increasing order so that when $y$ is colored,
if there are $\{x<y<z\}\in \H$, then $y$ gets the opposite color as $x$.
(This step is well defined, because it follows from the fact that the out-degree of $y$ is zero, that there is only one triple $\{x<y<z\}$ with the above property.)

This completes the proof of Theorem~\ref{thm:shiftchain}, as it follows that the middle vertex of any triple will have a different color from another vertex of the triple.
\end{proof}

\section{Covering space with unbounded convex sets}\label{sec:unbounded}

Every open, unbounded, line-free convex set $C$ is contained in a half-space with inner normal vector $\vec v$ such that for any $c\in C$, the half-line emanating from $c$ and pointing in the direction of $\vec v$ lies entirely in $C$. We can assume without loss of generality that $\vec v$ is the unit vector $e_d=(0,0,\ldots,0,1)$, pointing {\em vertically upwards}, and that $C$ lies in the {\em upper half-space}.

First, we prove Proposition~\ref{claim}, according to which every covering of $\R^d$ with translates of
a set $C$ satisfying the above conditions can be split into two, and hence into infinitely many, coverings.
We prove a slight generalization of this statement, in which $C$ is not required to be convex.

\medskip
\noindent{\bf Proposition~\ref{claim}'.} {\em
Let $C$ be an open set in the upper half-space of $\mathbb{R}^d$, which has the property that, for every $c\in C$, the half-line starting at $c$ and pointing vertically upwards belongs to $C$.
Then every covering of $\R^d$ with translates of $C$ splits into two coverings.}
\begin{proof}
For any positive integer $i$, let $B_i$ denote the {\em closed} $(d-1)$-dimensional ball of radius $i$ around the origin in the coordinate hyperplane $x_d=0$. Let $\C$ be a covering of $\R^d$ with translates of $C$. As the members of $\C$ cover the whole $d$-dimensional space, they also cover the $(d-1)$-dimensional ball $B_1\times \{0\}$, orthogonal to the $x_d$-axis. This set is {\em compact} and the members of $\C$ are {\em open}. Therefore, there is a {\em finite} subfamily $\C_1\subset \C$ which covers $B_1\times \{0\}$. Choose a number $z_1<0$ such that all members of $\C_1$ lie strictly above the hyperplane $x_d=z_1$, and consider the $(d-1)$-dimensional ball $B_2\times \{z_1\}$. Select a finite family $\C_2\subset\C$ that covers this ball and a number $z_2<z_1$ such that all members of $\C_2$ lie strictly above the hyperplane $x_d=z_2$. Proceeding like this, we can construct an infinite sequence of disjoint finite subfamilies $\C_1, \C_2,\ldots\subset\C$ and a sequence of reals $z_0:=0>z_1>z_2>\ldots$ tending to $-\infty$ such that $\C_i$ covers the $(d-1)$-dimensional ball $B_i\times\{z_{i-1}\}$.

Let $p$ be any point of $\mathbb{R}^d$ which is at distance $r$ from the $d^{th}$ coordinate axis and whose $d^{th}$ coordinate is $p_d$. Notice that $p$ lies above some point of every $(d-1)$-dimensional ball $B_i\times\{z_{i-1}\}$ such that $i\ge r$ and $z_{i-1}\le p_d$. Consequently, $p$ is covered by the corresponding families $\C_i$. Hence, $\C_1\cup\C_3\cup\C_5\cup\ldots$ and $\C_2\cup\C_4\cup\C_6\cup\ldots$ are two disjoint subfamilies of $\C$, each of which covers the whole space.
\end{proof}

Next, we establish Theorem~\ref{ugly}, which shows that starting from $4$-dimensions, Proposition~\ref{claim} is false if we drop the assumption that $C$ is an {\em open} set.

\begin{proof}[Proof of Theorem~\ref{ugly}.]
We have to prove that there is a convex, bounded (not open) set $C'\subset \R^3$ such that $\R^4$ can be covered by translates of $C=C'\times [0,\infty)$ so that every point of $\R^4$ is covered infinitely many times, but this covering cannot be decomposed into two.

The set $C'$ will be the convex hull of $\cup_{i=1}^\infty (C_i \times \{\frac 1{i^2}\})$, where each $C_i$ is in $\R^2$, and thus $C_i \times \{\frac 1{i^2}\}$ lies in the plane determined by the equation $z=\frac 1{i^2}$ of $\R^3$.
Each $C_i$ is the union of an open disk, defined by the inequality $x^2+(y-\frac 1i)^2<1$, and a part of its boundary defined as follows.
A point belongs to the boundary of $C_i$ if and only if it can be represented as $(x,\sqrt{1-x^2}+\frac 1i)$, where $x\in[0,1]$ and the $i^{th}$ digit of $x$ after the ``decimal'' point in binary form is $1$.
For each $i$, denote the set of such $x$'s by $C_i^*$.
Therefore, $C_i^*$ is the disjoint union of $2^{i-1}$ closed intervals.

Note that $C'$ is neither closed, nor open. Clearly, $C'$ is a bounded set, as it is contained in the box $[-1,1]\times [-1,2]\times [0,1]$.
Observe that for every $i$, the point $(0,\frac 1i, \frac 1{i^2})$, the center of the disk $C_i \times \{\frac 1{i^2}\}$, lies the plane $x=0$, on the parabola $z=y^2$.
Hence, for each $i$ and for every point $p\in\R^3$ whose third coordinate is $\frac 1{i^2}$ and first coordinate is non-negative, $p$ belongs to the boundary of $C'$ if and only if it is of the form $(x,\sqrt{1-x^2}+z,z^2)$ with $x\in C_i^*$ and $z=\frac 1i$. To see this, it is enough to notice that no point of this form can be obtained as a convex combination of other points in $C'$.

Now we describe an infinite-fold covering $\mathcal C$ of $\R^4$ with translates of $C$ that cannot be decomposed into two coverings.
Let $X=\{(x,\sqrt{1-x^2},0,-w)\mid x\in [0,1],w\in [0,\infty)\}$.
For every point $x\not\in X$, select an arbitrary translate of $C$ that covers $x$ and does not intersect $X$. (It is easy to see that such a translate always exists.) Let $\mathcal C$ consist of all these translates, and for every $i\; (i=1,2,\ldots)$, the translate of $C$ through the vector $(0,-\frac 1i, -\frac 1{i^2}, -i)$, denoted by $\hat C_i$.

Notice that the $\hat C_i$ covers $(x,\sqrt{1-x^2},0,-w)\in X$ if and only if $x\in C_i^*$ and $w\le i$.
This implies that every point of $X$ is covered by infinitely many members of $\mathcal C$, because every number has a representation with infinitely many digits that are $1$.

It remains to show that $\mathcal C$ cannot be split into two coverings.
This is a direct consequence of the following statement:
For any $I\subset \N$ for which $\N \setminus I$ is infinite, there is a point $(x,\sqrt{1-x^2},0,0)\in X$ that is not covered by $\cup_{i\in I} \hat C_i$.

To prove this statement when $I$ is {\em infinite}, define the $i^{th}$ digit of $x$ as $1$ if and only if $i\notin I$. Since this is only one binary representation of $x$, we have $x\notin\cup _{i\in I} \hat C_i^*$ and $(x,\sqrt{1-x^2},0,0)\notin\cup_{i\in I} \hat C_i$.
If $I$ is {\em finite}, it can be extended to an infinite set such that $\N \setminus I$ remains infinite. Thus, this case can be reduced to the case when $I$ is infinite.
\end{proof}

\section{Bounded coverings}\label{sec:lll}
We prove Theorem~\ref{thm:lll} in a somewhat more general form. For the proof we need the following consequence of the Lov\'asz local lemma.

\begin{lem}[Erd\hosszuo os-Lov\'asz~\cite{LLL}]\label{locallemma}
Let $k,m\ge 2$ be integers. If every edge of a hypergraph has at least $m$ vertices and every edge intersects at most $k^{m-1}/4(k-1)^m$ other edges, then its vertices can be colored with $k$ colors so that every edge contains at least one vertex of each color.
\end{lem}

Let $\mathcal C$ be a class of subsets of $\R^d$.
Given $n$ members $C_1,\ldots,C_n$ of $\mathcal C$, assign to each point $x\in \R^d$ a {\em characteristic vector} $c(x)=(c_1(x),\ldots,c_n(x))$, where $c_i(x)=1$ if $x\in C_i$ and $c_i(x)=0$ otherwise. The number of distinct characteristic vectors shows how many ``pieces''  $C_1,\ldots, C_n$ cut the space into. The {\em dual shatter function} of $\mathcal C$, denoted by $\pi^*_{\mathcal C}(n)$, is the maximum of this quantity over all $n$-tuples $C_1,\ldots,C_n\in\mathcal C$. For example, when $\mathcal C$ is the family of open {\em balls} in $\R^d$, it is well known that
\begin{equation}\label{gombok}
\pi^*_{\mathcal C}(n)\le{n-1\choose d}+\sum_{i=0}^d {n\choose i}\le n^d,
\end{equation}
provided that $2\le d\le n$.

\begin{thm}\label{shatter}
Let $\mathcal C$ be a class of open sets in $\R^d$ with diameter at most $D$ and volume at least $v$. Let $\pi(n)=\pi^*_{\mathcal C}(n)$ denote the dual shatter function of $\mathcal C$, and let $B^d$ denote the unit ball in $\R^d$. Then, for every positive integer $m$, any $m$-fold covering of $\R^d$ with members of $\mathcal C$ splits into two coverings, provided that no point of the space is covered more than $\frac{v}{(2D)^d Vol B^d}\pi^{-1}(2^{m-3})$ times, where $Vol B^d$ is the volume of $B^d$.
\end{thm}

\begin{proof}
Given an $m$-fold covering of $\R^d$ in which no point is covered more than $M$ times, define a hypergraph $\H=(V,E)$, as follows. Let $V$ consist of all members of $\mathcal C$ that participate in the covering. To each point $x\in\R^d$, assign a (hyper)edge $e(x)$: the set of all members of the covering that contain $x$. (Every edge is counted only once.) Since every point $x$ is covered by at least $m$ members of $\mathcal C$, every edge $e(x)\in E$ consists of at least $m$ points.

Consider two edges $e(x),e(y)\in E$ with $e(x)\cap e(y)\neq\emptyset$. Then there is a member of $\mathcal C$ that contains both $x$ and $y$, so that $y$ must lie in the ball $B(x,D)$ of radius $D$ around $x$. Hence, all members of the covering that contain $y$ lie in the ball $B(x,2D)$ of radius $2D$ around $x$. Since the volume of each of these members is at least $v$, and no point of $B(x,2D)$ is covered more than $M$ times, we obtain that $B(x,D)$ can be intersected by at most $M Vol B(x,2D)/v=M(2D)^d Vol B^d/v$ members of the covering. 
By the definition of the dual shatter functions, those members of the covering that intersect $B(x,D)$ cut $B(x,D)$ into at most $\pi(M(2D)^d Vol B^d/v)$ pieces, each of which corresponds to an edge of $\H$. Therefore, for the maximum number $N$ of edges of $\H$ that can intersect the same edge $e(x)\in E$, we have
$$N\le\pi(M(2D)^d Vol B^d/v).$$

According to Lemma~\ref{locallemma} (for $k=2$), in order to show that the covering can be split into two, i.e., the hypergraph $\H$ is $2$-colorable, it is sufficient to assume that $N\le 2^{m-3}$. Comparing this with the previous inequality, the result follows.
\end{proof}

In the special case where $\mathcal C$ is the class of unit balls in $\R^d$, we have $v=Vol B^d$, $D=2$, and, in view of (\ref{gombok}), $\pi^{-1}(z)\ge z^{1/d}$. Thus, we obtain Theorem~\ref{thm:lll} with $c_d=2^{-2d-3/d}$.

If we want to decompose an $m$-fold covering into $k>2$ coverings, then the above argument shows that it is sufficient to assume that
$$\pi(M(2D)^d Vol B^d/v)\le k^{m-1}/4(k-1)^m.$$
In case of unit balls, this holds for $M\le c_{k,d}(1+\frac1{k-1})^{m/d}$ with $c_{k,d}=k^{-1/d}4^{-d-1/d}$.

Two sets are {\em homothets} of each other if one can be obtained from the other by a dilation with positive coefficient followed by a translation. It is easy to see~\cite{G75} that for $d=2$, the dual shatter function of the class $\mathcal C$ consisting of all homothets of a fixed convex set $C$ is at most $n^2-n+2\le n^2$, for every $n\ge 2$. In this case, Theorem~\ref{shatter} immediately implies

\begin{cor}\label{homothet}
Every $m$-fold covering $\mathcal C$ of the plane with homothets of a fixed convex set can be decomposed into two coverings, provided that no point of the plane belongs to more than $2^{(m-11)/2}$ members of $\mathcal C$.
\end{cor}

Nasz\'odi and Taschuk \cite{NT} constructed a convex set $C$ in $\R^3$ such that the dual shatter function of the class of all translates of $C$ cannot be bounded from above by any polynomial of $n$. Therefore, for translates of $C$, the above approach breaks down. We do not know how to generalize Theorem~\ref{thm:lll} from balls to arbitrary convex bodies in $\R^d$, for $d\ge 3$.

For some related combinatorial results, see Bollob\'as et al.~\cite{B11}.




\section{Open problems and concluding remarks}\label{sec:open}

Theorem~\ref{thm:smooth} states that, if $C$ is a plane convex body with two antipodal points at which the curvature is positive, then for every $m$, there exists an $m$-fold covering of $\mathbb{R}^2$ with translates of $C$ that does not split into two coverings.
We also know that this statement is false for any convex polygon.
But what happens if $C$ ``almost satisfies'' the condition concerning the antipodal point pair?

\begin{problem}
Does there exist an integer $m$ such that every $m$-fold covering of $\mathbb{R}^2$ with translates of an open semidisk splits into two coverings?
\end{problem}


Another question, which surprisingly is widely open even in a completely abstract setting, is the following.

\begin{problem}
Suppose that for a body $C$, there is an integer $m$ such that every $m$-fold covering of $\mathbb{R}^d$ with translates of $C$ splits into two coverings. Does it follow that for every $k>2$, there is an integer $m_k$ such that every $m_k$-fold covering of $\mathbb{R}^d$ with translates of $C$ splits into $k$ coverings?
Is it true that (for the smallest such $m_k$) even $m_k=O_C(k)$?
\end{problem}

According to Theorem~\ref{thm:shiftchain}, for any $m\ge 3$, every $m$-uniform special shift-chain is $2$-colorable.
Keszegh and the P\'alv\"olgyi~\cite{KP14} recently extended this theorem to show that the vertices of every $(2k-1)$-uniform special shift-chain can be colored by $k$ colors so that every hyperedge contains at least one point of each color.

As was stated in the introduction, for every triangle (in fact, for every convex polygon) $C$, there is an integer $m(C)$ such that every $m$-fold covering of the plane with {\em translates} of $C$ splits into two coverings.
Keszegh and the P\'alv\"olgyi~\cite{KP} extended this theorem to $m$-fold coverings with {\em homothets} of a triangle. (Two sets are homothets of each other if one can be obtained from the other by a dilation with positive coefficient followed by a translation.) Using the idea of the proof of our Theorem~\ref{thm:disc}, Kov\'acs~\cite{K13} has recently showed that the analogous statement is false for homothets of any convex polygon with more than $3$ sides. For further results about decomposition of multiple coverings, see~\cite{A13,B11,colorful,colorful2,GV09,KP12,KP13}.

\subsection*{Acknowledgment}
The authors are deeply indebted to Professor Peter Mani, who passed away in 2013, for many interesting conversations about the topics, and his ideas reflected in the long unpublished manuscript~\cite{MP86}.
It was the starting point and an important source for the present work.

The authors also like to thank Radoslav Fulek, Bal\'azs Keszegh, and G\'eza T\'oth for their many valuable remarks.



\begin{thebibliography}{99}


\bibitem{Alon10} N. Alon, A non-linear lower bound for planar epsilon-nets, Discrete Comput. Geom. {\bf 47} (2012), no.~2, 235--244.

\bibitem{alonspencer} N. Alon\ and\ J. H. Spencer, {\it The probabilistic method}, third edition, Wiley-Interscience Series in Discrete Mathematics and Optimization, Wiley, Hoboken, NJ, 2008.

%
\bibitem{A09} G. Aloupis, J. Cardinal, S. Collette, S. Langerman, D. Orden, and P. Ramos, Decomposition of multiple coverings into more parts, Discrete Comput. Geom. {\bf 44} (2010), no.~3, 706--723.




\bibitem{A13} A. Asinowski, J. Cardinal, N. Cohen, S. Collette, T. Hackl, M. Hoffmann, K. Knauer, S. Langerman, M. Lason, P. Micek, G. Rote, and T. Ueckerdt, Coloring hypergraphs induced by dynamic point sets and bottomless rectangles, in {\it Algorithms and data structures}, 73--84, Lecture Notes in Comput. Sci., 8037, Springer, Heidelberg, 2013.





\bibitem{B11} B. Bollob\'as, D. Pritchard, T. Rothvo\ss, and A. Scott, Cover-decomposition and polychromatic numbers, SIAM J. Discrete Math. {\bf 27} (2013), no.~1, 240--256.


\bibitem{B07} A. L. Buchsbaum, A. Efrat, S. Jain, S. Venkatasubramanian, and K. Yi, Restricted strip covering and the sensor cover problem, in: {\em Proceedings of the Eighteenth Annual ACM-SIAM Symposium on Discrete Algorithms (SODA 2007)}, 1056--1063, ACM, New York, 2007.
%


\bibitem{colorful} J. Cardinal, K. Knauer, P. Micek, and T. Ueckerdt, Making triangles colorful, J. Comput. Geom. {\bf 4} (2013), no.~1, 240--246.

\bibitem{colorful2} J. Cardinal, K. Knauer, P. Micek, and T. Ueckerdt, Making octants colorful and related covering decomposition problems, SIAM J. Discrete Math. {\bf 28} (2014), no.~4, 1948--1959.




\bibitem{E63} P. Erd\hosszuo{o}s, On a combinatorial problem, Nordisk Mat. Tidskr. {\bf 11} (1963), 5--10, 40.

\bibitem{LLL}  P. Erd\hosszuo{o}s\ and\ L. Lov\'asz, Problems and results on $3$-chromatic hypergraphs and some related questions, in {\it Infinite and finite sets (Colloq., Keszthely, 1973; dedicated to P. Erd\hosszuo{o}s on his 60th birthday), Vol. II}, 609--627. Colloq. Math. Soc. J\'anos Bolyai, 10, North-Holland, Amsterdam.

\bibitem{ESz} P. Erd\hosszuo{o}s\ and\ G. Szekeres, A combinatorial problem in geometry, Compositio Math. {\bf 2} (1935), 463--470.

\bibitem{E02} G. Even, Z. Lotker, D. Ron, and S. Smorodinsky, Conflict-free colorings of simple geometric regions with applications to frequency assignment in cellular networks, SIAM J. Comput. {\bf 33} (2003), no.~1, 94--136.


\bibitem{FeH00}  U. Feige, M. M. Halld\'orsson, and G. Kortsarz, Approximating the domatic number, SIAM J. Comput. {\bf 32} (2002/03), no.~1, 172--195.


\bibitem{Fej83} G. Fejes T\'oth, New results in the theory of packing and covering, in {\it Convexity and its applications}, 318--359, Birkh\"auser, Basel.

\bibitem{FejK93} G. Fejes T\'oth\ and\ W. Kuperberg, A survey of recent results in the theory of packing and covering, in {\it New trends in discrete and computational geometry}, 251--279, Algorithms Combin., 10, Springer, Berlin.

\bibitem{F10} R. Fulek, personal communication, 2010. See also in \cite{PhD}.

\bibitem{ET} R. Fulek, T. Hubai, B. Keszegh, Z. Nagy, T. Rothvo\ss, and M. Vizer, unpublished. 

\bibitem{G09} H. Gebauer, H. Gebauer, Disproof of the neighborhood conjecture with implications to SAT, Combinatorica {\bf 32} (2012), no.~5, 573--587.


\bibitem{GV09} M. Gibson\ and\ K. Varadarajan, Optimally decomposing coverings with translates of a convex polygon, Discrete Comput. Geom. {\bf 46} (2011), no.~2, 313--333.


\bibitem{G75} B. Gr\"unbaum, Venn diagrams and independent families of sets, Math. Mag. {\bf 48} (1975), 12--23.

\bibitem{HW} D. Haussler\ and\ E. Welzl, $\epsilon$-nets and simplex range queries, Discrete Comput. Geom. {\bf 2} (1987), no.~2, 127--151.



%
\bibitem{KP} B. Keszegh\ and\ D. P\'alv\"olgyi, Octants are cover-decomposable, Discrete Comput. Geom. {\bf 47} (2012), no.~3, 598--609.

\bibitem{KP12} B. Keszegh\ and\ D. P\'alv\"olgyi, Octants are cover-decomposable into many coverings, Comput. Geom. {\bf 47} (2014), no.~5, 585--588.

\bibitem{KP13} B. Keszegh\ and\ D. P\'alv\"olgyi, Convex polygons are self-coverable, Discrete Comput. Geom. {\bf 51} (2014), no.~4, 885--895.

\bibitem{KP14} B. Keszegh and D. P\'alv\"olgyi, An abstract approach to polychromatic coloring: shallow hitting sets in ABA-free hypergraphs and pseudohalfplanes, arxiv:1410.0258.

\bibitem{K13} I. Kov\'acs, Indecomposable coverings with homothetic polygons, arXiv:1312.4597.

\bibitem{MP86} P. Mani-Levitska and J. Pach, Decomposition problems for multiple coverings with unit balls, manuscript, 1986. Parts of the manuscript are available at\\ \verb+http://www.math.nyu.edu/~pach/publications/unsplittable.pdf+

\bibitem{M11} J. Matou\v sek, The determinant bound for discrepancy is almost tight,  {\em Proc. Amer. Math. Soc.} {\bf  141} (2013), no.~2, 451--460.

\bibitem{Mi37} E. W. Miller, On a property of families of sets, C. R. Soc. Sci. Varsovie {\bf 30} (1937), 31--38.

\bibitem{NT} M. Nasz\'odi\ and\ S. Taschuk, On the transversal number and VC-dimension of families of positive homothets of a convex body, Discrete Math. {\bf 310} (2010), no.~1, 77--82.

\bibitem{P80} J. Pach, Decomposition of multiple packing and covering, {\em Diskrete Geometrie, 2. Kolloq.} Math. Inst. Univ. Salzburg, 1980, 169--178.

\bibitem{P86} J. Pach, Covering the plane with convex polygons, Discrete Comput. Geom. {\bf 1} (1986), no.~1, 73--81.
%
\bibitem{PPT} J. Pach, D. P\'alv\"olgyi\ and\ G. T\'oth, Survey on decomposition of multiple coverings, in {\it Geometry---intuitive, discrete, and convex}, 219--257, Bolyai Soc. Math. Stud., 24, J\'anos Bolyai Math. Soc., Budapest, 2013.
%
\bibitem{PTT09} J. Pach, G. Tardos\ and\ G. T\'oth, Indecomposable coverings, Canad. Math. Bull. {\bf 52} (2009), no.~3, 451--463.
%
\bibitem{PT09} J. Pach\ and\ G. T\'oth, Decomposition of multiple coverings into many parts, Comput. Geom. {\bf 42} (2009), no.~2, 127--133.


\bibitem{PhD} D. P\'alv\"olgyi, {\em Decomposition of geometric set systems and graphs}, PhD thesis, EPFL, Lausanne, 2010, arXiv:1009.4641.
%
\bibitem{P10} D. P\'alv\"olgyi, Indecomposable coverings with concave polygons, Discrete Comput. Geom. {\bf 44} (2010), no.~3, 577--588.
%
\bibitem{PT10} D. P\'alv\"olgyi\ and\ G. T\'oth, Convex polygons are cover-decomposable, Discrete Comput. Geom. {\bf 43} (2010), no.~3, 483--496.


\bibitem{RS} J. Radhakrishnan\ and\ A. Srinivasan, Improved bounds and algorithms for hypergraph 2-coloring. Random Structures Algorithms 16 (2000), no.~1, 4--32.

%

\bibitem{TT07} G. Tardos\ and\ G. T\'oth, Multiple coverings of the plane with triangles, Discrete Comput. Geom. {\bf 38} (2007), no.~2, 443--450.
%
\bibitem{V10} K. Varadarajan, Weighted geometric set cover via quasi-uniform sampling, in {\it STOC'10---Proceedings of the 2010 ACM International Symposium on Theory of Computing}, 641--647, ACM, New York.

\bibitem{W09} P. Winkler, on page 137 of {\it Mathematical mind-benders}, A K Peters, Wellesley, MA, 2007.
See also: P. Winkler, Puzzled: covering the plane, Commun. ACM {\bf 52} (2009), no.~11, 112.

\end{thebibliography}
\end{document}